\documentclass[12pt,a4paper]{amsart}

\usepackage[margin=3cm]{geometry}
\usepackage{graphicx} 
\usepackage{verbatim}
\usepackage{amssymb}
\usepackage{amsbsy}
\usepackage{amscd}
\usepackage{amsmath}
\usepackage{amsthm}
\usepackage{amsxtra}
\usepackage{latexsym}
\usepackage[mathscr]{eucal}
\usepackage{xcolor}
\usepackage{tikz}
\usetikzlibrary{arrows}

\theoremstyle{plain}
\newtheorem{thm}{Theorem}[section]

\newtheorem{prop}[thm]{Proposition}
\newtheorem{cor}[thm]{Corollary}
\newtheorem{conj}[thm]{Conjecture}

\theoremstyle{definition}
\newtheorem{dfn}[thm]{Definition}
\newtheorem{exa}[thm]{Example}

\newtheorem{COJ}[thm]{Conjecture}
\theoremstyle{remark}
\newtheorem{rmk}[thm]{Remark}

\newcommand{\fg}{\mathfrak g}
\newcommand{\C}{\mathbb{C}}
\newcommand{\SB}{S^\bullet}
\DeclareMathOperator{\End}{\mathrm{End}}

\newcommand{\V}{\mathcal{V}}
\newcommand{\OO}{\mathcal{O}}
\newcommand{\CP}{\mathbb{CP}}
\newcommand{\hFcl}{\hat{\mathcal{F}}_{cl}}
\newcommand{\hatA}{\hat{\mathcal A}}
\DeclareMathOperator{\Spec}{\operatorname{Spec}}

\newcommand{\W}{\mathcal{W}}

\newcommand{\Z}{\mathbb Z}

\DeclareMathOperator{\Span}{\mathrm{Span}}

\newcommand{\Cur}{\mathrm{Cur}}
\newcommand{\A}{\mathcal{A}}
\newcommand{\N}{\mathcal{N}}
\newcommand{\PD}{\mathcal{P}}

\newcommand{\BBA}{\mathbb{A}}
\newcommand{\Sf}{\mathscr{S}_f}
\newcommand{\fsl}{\mathfrak{sl}}
\newcommand{\cl}{\mathrm{cl}}
\newcommand{\sub}{\mathrm{sub}}
\newcommand{\fin}{\mathrm{fin}}

\newcommand{\Pois}{\mathrm{Poi}}

\newcommand{\ord}{\mathrm{ord}}
\newcommand{\cC}{\mathcal{C}}
\newcommand{\cM}{\mathcal{M}}

\newcommand{\cH}{\mathcal H}

\usepackage{hyperref}
\title{Shifted symplectic structures and Poisson vertex algebra}
\author{Wenda Fang}
\address{Research Institute for Mathematical Sciences, Kyoto University,
 Kyoto 606-8502 JAPAN}
\email{wenda@kurims.kyoto-u.ac.jp}
\thanks{}
\keywords{Representation theory, Algebraic geometry, Mathematical physics}

\begin{document}

\begin{abstract}
We identify Poisson vertex algebra (PVA) Hamiltonian structures on
arc spaces with Maurer--Cartan data arising from degree-$1$ graded
symplectic geometry.  Let $X$ be a smooth scheme of finite type and
set $Y=T^*[1]X$.  We prove that a degree-$2$ Hamiltonian on
$Y_\infty$ defines a sheaf of PVAs on $X_\infty$ precisely when its
class in the Lie algebra of local functionals satisfies the classical
master equation; conversely, every such PVA bracket determines a
unique degree-$2$ local functional.  We further show that classical
$R$-matrices are exactly Maurer--Cartan deformations of the
corresponding PVA Hamiltonian structure. As applications, we classify a class of global scalar PVA structures
on $\mathbb P^1$ and apply the $R$-matrix formalism to finite and
affine classical $\mathcal W$-algebras.  For the subregular
$\mathfrak{sl}_3$ example, we construct a compatible Poisson pencil
and the centrally reduced RDW system, whose generic reduction is a
genus-two hyperelliptic system admitting an explicit Mumford-type Lax
representation.  Its affine counterpart yields constant and
first-order differential $R$-deformations, together with commuting
Hamiltonian flows.
\end{abstract}

\maketitle

\section{Introduction}

Poisson vertex algebras (PVAs) provide a purely algebraic framework for Hamiltonian
structures of evolutionary PDEs~\cite{BDSK09,FBZ04}. A key advantage of the PVA formalism is
its algebraic nature, which makes it well suited to symbolic computation~\cite{CV18}.
Analogous to the way Poisson algebras arise as the $\hbar\to 0$ limits of one-parameter
families of associative algebras $A_{\hbar}$, a PVA may be viewed as the classical limit
of a family of vertex algebras $V_{\hbar}$~\cite{K1,Li04}.

The purpose of this paper is to describe Poisson vertex algebra (PVA)
structures on arc spaces in graded-symplectic and Maurer--Cartan terms.
We call this passage from graded-symplectic data to arc-space structures
\emph{QP chiralization}, and its degree-$n$ version
\emph{degree-$n$ shifted chiralization}. In the degree-$1$ case, let
$X$ be a smooth scheme of finite type over $\mathbb C$ and set
$Y:=T^*[1]X$. The shifted cotangent bundle $Y$ carries a canonical
degree-$1$ symplectic structure. On the arc space $Y_\infty$,
degree-$2$ functions encode variational bivectors and induce
$\lambda$-brackets on $\mathcal O_{X_\infty}$. The PVA Jacobi identity
is then equivalent to the classical master equation for the
corresponding class in the Lie algebra of local functionals.

\medskip
Our first main result is the following correspondence theorem.

\begin{thm}[Main Theorem 1 (Correspondence Theorem)]\label{Thm:CS}
Let \(X\) be a smooth, reduced, finitely generated \(\C\)-scheme, and set \(Y:=T^*[1]X\).
Here \(\mathscr H\) denotes the sheaf of local functionals on \(X\), see Theorem~\ref{Thm:CS}.
Then a degree-$2$ element \(P\in\Gamma(Y_\infty,\OO_{Y_\infty})_2\) defines a Poisson vertex algebra structure on \(X_\infty\)
(i.e.\ it makes \(\OO_{X_\infty}\) a sheaf of PVAs) if and only if
\[
\bigl[\!\int P,\!\int P\bigr]=0,
\]
where \([\,\cdot\,,\,\cdot\,]\) denotes the Lie bracket on local functionals induced by the chiral Gerstenhaber bracket.

Conversely, any sheaf PVA \(\lambda\)-bracket on \(\OO_{X_\infty}\) determines a unique degree-\(2\) section of \(\mathscr H\),
which we denote by \(\int P\), such that
\[
\bigl[\!\int P,\!\int P\bigr]=0.
\]
\end{thm}

\medskip
Our second main result identifies classical $R$-matrices with Maurer-Cartan elements in the same graded-geometric
framework. Let \((L,\partial,[\,\cdot{}_{\lambda}\cdot\,])\) be a Lie conformal algebra such that \(L\) is a free
\(\C[\partial]\)-module of finite rank \(N\), and let \(\V:=\SB(L)\) be the free commutative differential algebra generated by \(L\),
endowed with the PVA structure obtained by extending the \(\lambda\)-bracket via the Master Formula.
Then \(\V\simeq \OO\bigl((\BBA^N)_\infty\bigr)\), hence we may apply Theorem~\ref{Thm:CS} to \(X=\BBA^N\) and view the induced
PVA structure as a degree-$2$ Hamiltonian \(P\) on the arc space \((T^*[1]\BBA^N)_\infty\).

\begin{thm}[Main Theorem 2]\label{Thm:RM}
Assume \(R\in\End_{\C}(L)\) satisfies \([R,\partial]=0\). Define the degree-$2$ element
\[
P_R \;=\; \frac{1}{2}\theta_\alpha\,\{\,u^{\beta}{}_{\partial}u^{\alpha}\,\}_{R,\rightarrow}\,\theta_\beta
\ \in\ \hatA^{2},
\]
where \({}_{\rightarrow}\) means: replace \(\lambda\) by \(\partial\) and move all \(\partial\)'s to the right.
Then $R$ is a classical $R$-matrix if and only if
\[
\bigl[\!\int P_R,\!\int P_R\bigr]=0.
\]
Equivalently, writing \(\Theta_R:=\int(P_R-P)\), the element \(\Theta_R\) is a Maurer-Cartan element of the dg Lie algebra
\((\hFcl,[\,\cdot\,,\,\cdot\,],d_P:=[\int P,-])\), i.e.\ \(d_P\Theta_R+\tfrac12[\Theta_R,\Theta_R]=0\).
\end{thm}
\medskip
As a first application, we study sheaves of PVAs on $X_\infty$ and the global extendability problem.
For $X=\mathbb P^1$, within the class considered in Section~5.1, we obtain Virasoro--Magri and constant-coefficient brackets when $c\ne0$, and hydrodynamic-type brackets with $q\in H^0(\mathbb P^1,\mathcal O(4))$ when $c=0$; see Section~\ref{Sec:SPVA}.

\medskip
 Motivated by this framework, we extend the notion of a classical $R$-matrix to finite classical $\W$-algebras; in the affine case, this yields a natural QP-chiralization within our framework. In particular, we obtain an integrable hierarchy in the case of $\W^{\cl}_{\fin}(\fsl_3,f_{\sub})$. 

In analogy with the classical work of Belavin--Drinfel'd~\cite{BD82} on constant solutions of the classical Yang--Baxter equation, we classify constant $R$-deformations within three natural families; see Proposition~\ref{prop:R-types}. In the mixed-type case, we construct a bi-Hamiltonian system; see Proposition~\ref{prop:R-types} and Remark~\ref{rmk:RDWQV}. Building on this example, we
discuss potential connections with Arakawa's 4d/2d program
(associated varieties/Higgs branches and class $\mathcal S$ chiral algebras)~\cite{A18ICM,A18,A19CA}.
In particular, the mixed-type $R$-deformation produces the
three-dimensional polynomial system (13), with conserved quantities
\[
S=\frac16u_2^3+u_1u_3,
\qquad
J=u_1^2+u_3^2.
\]
On the common level set $S=s$ and $J=j$, the functions
$x:=u_2$ and $y:=3(u_1^2-u_3^2)$ satisfy
\[
y^2=-x^6+12sx^3+9(j^2-4s^2).
\]
This system also admits the Lax representation of Proposition~5.9.
We propose two conjectures concerning its geometric interpretation;
see Conjectures~5.12 and~5.14. Related questions and possible connections to
integrable hierarchies are being explored in discussions with T.~Shiota.\\ 

\begin{COJ}\label{COJ:RDW1}
Let $\nabla=d-L(\lambda)\,d\lambda$ be the meromorphic connection attached to~\eqref{eq:RDW}, and let $C$ be its (possibly ramified) formal irregular type at $\infty$. We expect that a generic symplectic leaf $S=s$ of the original Poisson structure $H$, carrying the RDW flow, is birationally Poisson equivalent to a natural symplectic slice (or, more generally, to a coisotropic reduction) of a wild Dolbeault/Hitchin moduli space associated with $C$. Under this identification, a suitable locally defined function of the Hitchin Hamiltonians induces the Hamiltonian $J/2$ on the leaf and generates the RDW flow.
\end{COJ}

\medskip\noindent\textit{Evidence.}
Assuming the corresponding stable wild character variety is
nonempty, Remark~\ref{rmk:RDWQV} gives it dimension four, whereas a generic symplectic leaf $S=s$ of $H$ has dimension two. This
dimension difference motivates the slice or reduction appearing
in Conjecture~\ref{COJ:RDW1}. Corollary~\ref{cor:GCM} computes the Boalch--Yamakawa core diagram~\cite{BY20}
(two nodes with three edges, $C=\bigl[\begin{smallmatrix}2&-3\\-3&2\end{smallmatrix}\bigr]$); Proposition~\ref{Prop:LAX} provides an isospectral Mumford-type Lax representation.

\begin{COJ}\label{COJ:RDW2}
Assume Conjecture~\ref{COJ:RDW1}.
Then the associated Seiberg--Witten data are
\[
\Sigma:\mu^2=g(\lambda), \qquad \lambda_{\rm SW}=\mu\,d\lambda,
\]
and are compatible with the $(A_1,A_5)$ normal form after the standard rescaling and translation.
Assume further that this is a subfamily of the Seiberg--Witten geometry
of the rank-two four-dimensional $\mathcal N=2$
$(A_1,A_5)$ Argyres--Douglas theory $\mathcal T$.
Then the $4d/2d$ correspondence predicts
\[
\chi(\mathcal T)\cong \W^{(2)}_3
\]
at level $k=-9/4$.
\end{COJ}

\medskip\noindent\textit{Evidence.}
 Comparing $\mu^2=g(\lambda)$ with the standard $(A_1,A_{N-1})$ Seiberg--Witten normal form
in Xie's classification~\cite{X13GAD} yields $N=6$ (and hence $[\lambda]=\tfrac14$, $[\mu]=\tfrac34$);
moreover, the physics literature identifies the $A_5$ chiral algebra with $\W^{(2)}_3$
at level $k=-9/4$ and supports this by matching the $\W^{(2)}_3$ vacuum character with
the $A_5$ Schur index (we do not repeat this check here)~\cite{CGS17SDCA}.
\medskip

For related folding results in the setting of ordinary Hitchin systems,
see Beck--Donagi--Wendland~\cite{BDW22}.
\medskip

The arc space construction developed here fits naturally into the
graded-symplectic description of Poisson geometry. Degree-$1$ graded
symplectic manifolds equipped with Hamiltonians satisfying the classical
master equation provide the QP (equivalently, symplectic NQ) formulation
underlying the AKSZ formalism and the Poisson sigma model~\cite{AKSZ97}.
Higher-degree analogues include Roytenberg's graded-symplectic
description of Courant algebroids~\cite{Roy02}. The QP-manifold approach
to current algebras of Ikeda--Xu~\cite{IX14,IX25} offers another related
perspective; its precise relation to our arc space construction remains
to be clarified. More broadly, shifted symplectic geometry suggests
organizing Poisson-type structures through master equations~\cite{PTVV13}.

Theorems~\ref{Thm:CS} and~\ref{Thm:RM} realize this perspective on arc
spaces: PVA structures and their $R$-deformations are expressed through
Maurer--Cartan equations in the Lie algebra of local functionals. In
particular, the classical $R$-matrices of~\cite{Fang25} are realized as
Maurer--Cartan elements in the corresponding dg Lie algebra, and the
associated canonical transformation formalism yields AKS-type
hierarchies. At the level of local functionals, this description relates
the $\lambda$-bracket/Master Formula formalism of
PVAs~\cite{BDSK09,K1} to the variational Schouten/Jacobi approach to
Hamiltonian structures and their deformations. It is closely related to
Getzler's formal calculus of variations~\cite{Get02}, Liu--Zhang's
supergeometric treatment of Jacobi structures~\cite{LZ11}, and the
Dubrovin--Zhang program~\cite{DZ01}. It also provides a setting for
Miura-type transformations and accords with the role of arc schemes in
Arakawa's work on $\W$-algebras.

Several further directions arise from this viewpoint. Our interest in chiralizing Poisson--Lie structures was inspired in part by Safronov's work~\cite{Saf17,Saf19,Saf21} on shifted Poisson geometry and by discussions with N.~Ikeda.

Related work also points to questions of deformation and quantization.
Butson--Nair's work on chiral quantizations of vertex Poisson algebras
on arc spaces~\cite{BN26} provides a natural point of comparison with
our local-functional Maurer--Cartan description. Khan--Zeng's
three-dimensional holomorphic--topological gauge theory associated
with freely generated PVAs~\cite{KZ25} offers a field-theoretic
perspective on PVA quantization. At the level of integrable systems,
Wang's quantum dispersionless KdV hierarchy~\cite{Wang24} provides a
further example of quantization.

\medskip\noindent
\textit{Organization of the paper.} Section~2 reviews graded algebras and graded geometry. Section~3 introduces graded Lie conformal algebras, chiral Gerstenhaber algebras, and local functionals. Section~4 constructs the $\lambda$-bracket from $P$ on arc spaces and proves Theorem~\ref{Thm:CS}. Section~5 presents the global $\CP^1$ application and the Maurer--Cartan description of classical $R$-matrices (Theorem~\ref{Thm:RM}). It also develops applications to classical $\mathcal W$-algebras, including the centrally reduced subregular $\mathfrak{sl}_3$ RDW system, its geometric conjectures, and its chiralization with commuting Hamiltonian flows.

Throughout this paper, we consider a reduced smooth scheme $X$ of finite type over $\C$.

{\it Acknowledgements.}
The author is deeply grateful to Tomoyuki Arakawa for his invaluable guidance, encouragement, and many suggestions that substantially improved this work. The author especially thanks Takahiro Shiota for insightful discussions and numerous comments. The author thanks N.~Ikeda for his interest in this project and for helpful discussions on Poisson geometry. The author is indebted to Victor Kac for helpful conversations and to Alberto De Sole for careful discussions on Poisson vertex algebras. The author thanks Zhe Wang for weekly discussions and for patiently explaining his formalism, and thanks Matteo Casati for guidance on geometric aspects. The author also thanks Muze Ren for kind suggestions and for discussions in Geneva and Osaka. The author is grateful to Anne Moreau for discussions in Luminy and Saclay, and for her interest in this work. He also thanks Shoma Sugimoto for reading an earlier version of this preprint and for providing helpful suggestions. He thanks the Department of Mathematics at the University of Arizona for the invitation to present further joint work with N.~Ikeda. The author is grateful to the organizers of the workshops on vertex algebras and integrable systems held under the EPSRC grant ``Integrable models and deformations of vertex algebras via symmetric functions'', as well as to the organizers of Representation Theory XIX, for their organization and hospitality. The author also thanks the organizers of ``Poisson geometry and related topics 25'' and the RIMS workshop ``Research on finite groups, algebraic combinatorics, and vertex algebras'' for the opportunity to present this work. This work was supported by JST SPRING (Grant Number JPMJSP2110) and by JSPS KAKENHI (Grant Number JP21H04993). Finally, the author thanks Bohan Li, Hao Li, Ana Kontrec, and other members of Tomoyuki Arakawa's research group for their encouragement.

\section{Preliminary: Graded manifold and graded algebra}\label{Pre:I}
In this section, we recall some known results about supermanifolds and superalgebras.

\subsection{Graded algebra}\label{Sec:GA}
We recall the graded manifold and the graded algebra here. The main references are \cite{CFL06,CS11,QZ11}.

\begin{dfn}\label{dfn:GVS}
A $\Z$-graded vector space $V$ is a direct sum $V=\bigoplus_{i\in\Z}V_i$ over the complex numbers $\C$.
The $V_i$'s are called the components of $V$ of degree $i$ and the degree of a homogeneous element $a\in V$ is denoted by $|a|$.
\end{dfn}

For graded vector spaces $V, W$ the tensor product of them is a graded vector space whose degree component $r$ is given by
$(V\otimes W)_r:=\bigoplus_{p+q=r}V_p\otimes W_q.$
We can also define the exterior algebra and the symmetric algebra of a graded vector space $V$ as follows.

\begin{dfn}\label{dfn:symex}
The symmetric and exterior algebras of a graded vector space $V$ are defined respectively as
\[
S^\bullet(V):=T(V)/I_s,\qquad \bigwedge\nolimits^{\bullet}(V):=T(V)/I_\wedge,
\] 
where $T(V)=\bigoplus_{n\geq0}V^{\otimes n}$ is the tensor algebra of $V$ and $I_s$ (resp.\ $I_\wedge$) is the two-sided ideal generated by the elements of the form
$a\otimes b-(-1)^{|a||b|}b\otimes a$ (resp.\ $a\otimes b+(-1)^{|a||b|}b\otimes a$).
The images of $V^{\otimes n}$ in $S^\bullet(V)$ and $\bigwedge^\bullet(V)$ are denoted by $S^n(V)$ and $\bigwedge^n(V)$, respectively.
\end{dfn}

Let $V$ be a graded vector space over $\C$. For an integer $n$ we write $V[n]$ for the degree-shift of $V$; explicitly,
$V[n]:=\bigoplus_{i\in\Z} (V[n])_{i}$ where $(V[n])_i:=V_{n+i}$.

\begin{dfn}\label{dfn:GLA}
A graded Lie algebra (GLA) of degree $n$ is a graded vector space $A$ endowed with a graded Lie bracket on $A[n]$ such that the bracket can be seen as a degree $-n$ Lie bracket on $A$, i.e., as a bilinear operation
\begin{equation}\label{eq:GLA1}
\{-,-\}:A\times A\rightarrow A[-n],
\end{equation}
such that graded anti-symmetry and the graded Jacobi identity hold:
\begin{align}\label{eq:GLA2}
&\{a,b\}=-(-1)^{(|a|+n)(|b|+n)}\{b,a\},\\
&\{a,\{b,c\}\}=\{\{a,b\},c\}+(-1)^{(|a|+n)(|b|+n)}\{b,\{a,c\}\},
\end{align}
for all homogeneous $a,b,c\in A.$
\end{dfn}

In particular, a graded Lie algebra whose bracket has degree zero becomes a Lie superalgebra~\cite{Kac771} upon reducing the grading modulo $2$.
Note that there exists a canonical d\'ecalage isomorphism
$S^n\bigl(V[1]\bigr)\cong\bigl(\bigwedge^n V\bigr)[n].$

\begin{dfn}\label{dfn:GPA}
A graded Poisson algebra of degree $n$ (or an $n$-Poisson algebra) is a triple $(A,\circ,\{-,-\})$ consisting of a graded vector space
$A=\bigoplus_{i\in\Z}A_i$ endowed with a degree-zero graded associative commutative product and a Lie bracket of degree $-n.$
For homogeneous $a,b,c\in A$ the bracket satisfies the graded Leibniz rule
$\{a,b\circ c\}=\{a,b\}\circ c+(-1)^{|b|(|a|+n)}\,b\circ\{a,c\}.$
When $n=1$, such an algebra is called a Gerstenhaber algebra.
A standard example is the algebra of multivector fields, for which
the degree-$(-1)$ bracket is the Schouten--Nijenhuis bracket.
When $n=0$, one recovers the usual notion of a graded Poisson algebra.
\end{dfn}

For the $\Z_2$-graded case, one simply says even or odd Poisson algebra. In this situation, an even Poisson algebra coincides with a Poisson superalgebra.

\begin{exa}\label{exa:ShiftPA}
Let $\mathfrak{g}$ be a graded Lie algebra of degree $0$.
Put $A:=S^\bullet(\mathfrak{g}[-n])$ whose multiplication is the one naturally induced from the tensor algebra $T(\mathfrak{g}[-n])$.
The Poisson bracket is specified as follows.
On the linear component $S^1(\mathfrak{g}[-n])$ it is the suspension of the bracket on $\mathfrak{g}$; for homogeneous elements $a,b,c\in A$, we extend it by the graded Leibniz rule
\[
\{a,b\circ c\}=\{a,b\}\circ c+(-1)^{|b|(|a|+n)}\,b\circ\{a,c\}.
\]
Hence $S^\bullet(\mathfrak{g}[-n])$ is an $n$-Poisson algebra.
\end{exa}

\begin{exa}\label{exa:osp}
Consider the Lie superalgebra $\mathfrak g=\mathfrak{osp}(1|2)$ with its standard compatible $\mathbb Z$-grading. Then $S^\bullet(\mathfrak g[-1])$ has a Gerstenhaber algebra structure.
\end{exa}

By shifting the degree of a Gerstenhaber algebra $(V^\bullet,\{-,-\},\cdot)$, we can reformulate its definition as that of a graded Lie algebra equipped with a compatible Leibniz rule (see \cite{LGCP}). In Section \ref{Sec:SSSPVA}, we follow the conventions of \cite{DZ01}, which are standard in the theory of integrable systems.

\begin{rmk}\label{rmk:GeA}
By shifting the degree of the vector space, for a Gerstenhaber algebra $(V^\bullet,\{-,-\},\cdot)$,
we define $[a,b]:=(-1)^{|a|-1}\{a,b\}$. Then the graded skew-symmetry and graded Leibniz rule can be written as
\[
[b,a]=(-1)^{|a||b|}[a,b],\qquad
[c,a\cdot b]=[c,a]\cdot b+(-1)^{|a|(|c|+1)}\,a\cdot[c,b].
\]
See \cite[(2.1.13)-(2.1.14)]{DZ01,Get02} and \cite[Section 3.3.2]{LGCP}.
\end{rmk}

\begin{rmk}
This shows that one can regard the Gerstenhaber algebra as a generalization of graded Poisson algebra.
\end{rmk}

\subsection{Graded manifold}
In this section, we recall basic notions on graded manifolds and graded schemes.
Our main references are \cite{QZ11,CFL06,CS11}. For connections to vertex algebras and Poisson vertex algebras, see \cite{Ara12,Li21}.

\begin{dfn}
A $\Z$-graded scheme is a pair $(X,\mathcal O_X=\bigoplus_{n\in\Z}\mathcal O_X^n),$
where $(X,\OO_X^0)$ is a scheme and $\OO_X$ is a sheaf
of graded-commutative graded $\OO_X^0$-algebras.

Moreover, there exists an affine open cover $
X=\bigcup_i U_i,\ U_i=\Spec A_i,$
and isomorphisms of graded $\OO_{U_i}$-algebras $\OO_X|_{U_i}\simeq\widetilde{B^{(i)}},$
where $B^{(i)}=\bigoplus_{n\in\Z}B^{(i),n}$
is a finitely generated $\Z$-graded, graded-commutative
$A_i$-algebra such that $B^{(i),0}\simeq A_i.$
\end{dfn}

\begin{exa}\label{SS1}
Let $R^{n|m}
=
\C[x^1,\ldots,x^n]
\otimes
\Lambda(\theta_1,\ldots,\theta_m),$
where $|x^i|=0,\ |\theta_j|=1.$
Then $(R^{n|m})^0=\C[x^1,\ldots,x^n].$
Hence the underlying scheme is $X=\Spec (R^{n|m})^0=\BBA^n,$
with $\OO_X^0=\OO_{\BBA^n},\ 
\OO_X
=
\OO_{\BBA^n}
\otimes
\Lambda(\theta_1,\ldots,\theta_m).$
We denote this graded scheme by $\BBA^{n|m}:=(X,\mathcal O_X).$
\end{exa}

\begin{exa}\label{exa:TB}
A basic example of a graded (and hence super) manifold is the shifted
tangent bundle $T[1]X$.
Let $X$ be a finite-dimensional smooth manifold with local coordinates
$x^i$, $i=1,\dots,n$.
On $T[1]X$ we use local coordinates $(x^i,\theta^i)$, where $x^i$ have
degree $0$ and $\theta^i$ have degree $1$ (hence are odd).
On an overlapping chart with coordinates $\tilde x^i = \tilde x^i(x)$, the
gluing rules are
\[
  \tilde x^i = \tilde x^i(x), \qquad
  \tilde\theta^i = \frac{\partial \tilde x^i}{\partial x^j}\,\theta^j,
\]
so that the $\theta^i$ transform as the local basis $dx^i$ of one-forms.
Consequently the algebra of global functions on $T[1]X$ is naturally
identified with the graded commutative algebra of differential forms
$\Omega^\bullet(X)$, where a $k$-form has degree $k$.
\end{exa}
Similar construction works for the cotangent bundle.

\begin{exa}\label{exa:Tstar1X}
Similarly, consider the shifted cotangent bundle $Y := T^*[1]X$.
On a smooth manifold $X$ with local coordinates $(x^\mu)$, we use
coordinates $(x^\mu,\theta_\mu)$ on $T^*[1]X$, where the $x^\mu$ have
degree $0$ and the $\theta_\mu$ have degree $1$.
On an overlapping chart with coordinates $\tilde x^\mu = \tilde x^\mu(x)$,
the gluing is given by the degree-preserving maps
\[
  \tilde x^\mu = \tilde x^\mu(x), \qquad
  \tilde\theta_\mu = \frac{\partial x^\nu}{\partial\tilde x^\mu}\,\theta_\nu.
\]
Locally, the coordinate functions generate a graded-commutative algebra
$\Gamma(Y,\OO_Y)$, whose degree-$k$ part identifies with the space
$\Gamma(X,\bigwedge\nolimits^{k} T_X)$ of $k$-vector fields on $X$.
In particular,
\[
  \Gamma(T^*[1]X,\OO_{T^*[1]X})
  \;\cong\; \Gamma\bigl(X,\bigwedge\nolimits^{\bullet} T_X\bigr).
\]
\end{exa}

\begin{exa}\label{exa:CDO}(see \cite[Example 3.5]{QZ11} and \cite[Letter~7]{severa2017})
Consider the graded cotangent bundle $T^*[2](T^*[1]X)$ and let
$Z := T^*[2](T^*[1]X)$.
Let $(x^\mu,\theta^\mu,\psi_\mu,p_\mu)$ be local coordinates of degree
$0,1,1,2$, respectively.
On an overlapping chart with coordinates $\tilde x^\mu = \tilde x^\mu(x)$,
the gluing between patches is given by the degree-preserving maps
\begin{align*}
  &\tilde x^\mu=\tilde x^\mu(x), \qquad
   \tilde\theta^\mu=\frac{\partial\tilde x^\mu}{\partial x^\nu}\,\theta^\nu,
   \qquad
   \tilde\psi_\mu=\frac{\partial x^\nu}{\partial\tilde x^\mu}\,\psi_\nu,\\
  &\tilde p_\mu
   =\frac{\partial x^\nu}{\partial\tilde x^\mu}\,p_\nu
    +\frac{\partial^2 x^\nu}{\partial\tilde x^\gamma\,\partial\tilde x^\mu}
     \,\frac{\partial\tilde x^\gamma}{\partial x^\sigma}\,
     \psi_\nu\theta^\sigma.
\end{align*}
Locally, the coordinate functions generate the graded-commutative algebra $\Gamma(Z,\OO_Z)$.
Its degree-$0$ part identifies with $\Gamma(X,\OO_X)$, and its degree-$1$ part with $\Gamma(X,T_X\oplus T_X^*)$
(cf. \cite[Example~3.5]{QZ11}).
Equipped with its canonical symplectic form of degree $2$ and a degree-$3$ Hamiltonian $\Theta$ satisfying $\{\Theta,\Theta\}=0$,
such a graded manifold encodes a Courant algebroid \cite[Letter~7]{severa2017,Roy02}.
\end{exa}

\section{Graded conformal algebras}\label{Sec:GCA}
In this section we review graded (Lie) conformal algebras and introduce chiral Gerstenhaber algebras. 

\subsection{Graded Lie conformal algebra}\label{Sec:PVAGO}
For completeness we now define a graded Lie conformal algebra $L$.
It may be interesting to compare this algebra with the dg vertex Lie algebra introduced by Caradot, Jiang, and Lin (see \cite{AJL24}).

\begin{dfn}\label{dfn:GLCA}
A graded Lie conformal algebra is a $\Z$-graded
$\C[\partial]$-module
\[
L=\bigoplus_{i\in\mathbb Z}L_i,
\]
such that $\partial$ has degree $0$, together with a
$\mathbb C$-bilinear operation
\[
[-_\lambda-]\colon
L\otimes L\longrightarrow \mathbb C[\lambda]\otimes_{\C}L,
\]
\[
a\otimes b\longmapsto
[a_\lambda b]
:=
\sum_{j\geq 0}\frac{\lambda^j}{j!}(a_{(j)}b)
=
\sum_{j\geq0}\lambda^{(j)}(a_{(j)}b),
\]
of degree $0$, namely $[L_i{}_\lambda L_j]\subset
\mathbb C[\lambda]\otimes L_{i+j},$
satisfying the following axioms for all homogeneous $a,b,c\in L$:
\begin{enumerate}
\item Sesquilinearity: $[\partial a_\lambda b]
=
-\lambda[a_\lambda b],\ 
[a_\lambda\partial b]
=
(\lambda+\partial)[a_\lambda b].$

\item Skew-symmetry: $[b_\lambda a]
=
-(-1)^{|a||b|}
[a_{-\lambda-\partial}b].$

\item Jacobi identity: $[a_\lambda[b_\mu c]]
=
[[a_\lambda b]_{\lambda+\mu}c]
+
(-1)^{|a||b|}
[b_\mu[a_\lambda c]].$
\end{enumerate}
Here $\lambda$ is an even variable, and the signs are determined by
the parity induced from the $\mathbb Z$-grading modulo $2$.
\end{dfn}

\begin{dfn}\label{dfn:GLCAN}
Similar to the graded Lie algebra case, a graded LCA of degree $n$ is a graded
$\mathbb C[\partial]$-module $L$ endowed with a $\mathbb C$-bilinear operation
\[
\{-_\lambda-\}:L\otimes L\to L[-n][\lambda]
\]
(with $\deg\lambda=0$), satisfying the same sesquilinearity relations as in
Definition~\ref{dfn:GLCA} and, for all homogeneous $a,b,c\in L$,
\begin{align*}
&\{a_\lambda b\}
=-(-1)^{(|a|+n)(|b|+n)}
\{b_{-\lambda-\partial}a\},\\
&\{a_\lambda\{b_\mu c\}\}
-(-1)^{(|a|+n)(|b|+n)}
\{b_\mu\{a_\lambda c\}\}
=\{\{a_\lambda b\}_{\lambda+\mu}c\}.
\end{align*}
Here $\{L_i{}_\lambda L_j\}\subseteq L_{i+j-n}[\lambda]$ for all $i,j\in\Z$.
\end{dfn}

\begin{exa}\label{exa:SCur}{\cite[Example 3.3]{FK}}
Let $\fg$ be a finite-dimensional Lie superalgebra.
Then $\Cur(\fg)=\C[\partial]\otimes\fg$ is a Lie conformal superalgebra with
$\lambda$-bracket $[a_\lambda b]=[a,b]$ for all $a,b\in\fg$, and we extend the
$\lambda$-bracket to $\Cur(\fg)$ by sesquilinearity.
\end{exa}

Here, we define a chiral Gerstenhaber algebra. It is interesting to compare this definition with the one given by Tamarkin in his ICM talk (see \cite{Ta02}), as well as with those in \cite{DSK13} and \cite{BDSH19}.

\begin{dfn}\label{dfn:CGPA}
A chiral Gerstenhaber algebra is a quintuple $(\V^\bullet,1,\cdot,\{-_{\lambda}-\},\partial)$ such that $\V^\bullet$ is a $\Z$-graded vector space, and $\partial$ is a degree $0$ endomorphism with respect to this $\Z$-grading. In addition, the product $\cdot$ is a degree $0$ graded commutative product and the bilinear operator $\{-_\lambda-\}$ is a degree $-1$ $\lambda$-bracket. Moreover, this quintuple satisfies the following axioms.
\begin{itemize}
\item[1.] $(\V^\bullet,1,\partial,\cdot)$ is a commutative, associative differential algebra whose product has degree $0$.\\
\item[2.] $(\V^\bullet,\partial,\{-_{\lambda}-\})$ is a graded Lie conformal algebra of degree $1$.
\begin{equation}\label{eq:SGJ}
\{a_{\lambda}\{b_{\mu}c\}\}=\{\{a_{\lambda}b\}_{\lambda+\mu}c\}+(-1)^{(|a|+1)(|b|+1)}\{b_{\mu}\{a_{\lambda}c\}\}.
\end{equation}\\
\item[3.] $\{-_{\lambda}-\}$ and $\cdot$ satisfy the following shifted Leibniz rule:
$\{a_{\lambda}b\cdot c\}=\{a_{\lambda}b\}c+(-1)^{(|a|+1)|b|}b\{a_{\lambda}c\}.$
\end{itemize}
We shall also call its $\lambda$-bracket the chiral Nijenhuis-Schouten bracket.
\end{dfn}

\begin{rmk}\label{rmk:CGPA}
Using the same shifted sign conventions as in Remark~\ref{rmk:GeA}, by shifting the grading on $\V^\bullet$, the graded skew-symmetry and the odd Leibniz rule may be written as
\[
\{b_{\lambda}a\}=-(-1)^{(|a|+1)(|b|+1)}\,\{a_{-\lambda-\partial}b\},\qquad
\{c_{\lambda}(a\cdot b)\}=\{c_{\lambda}a\}\cdot b+(-1)^{(|c|+1)|a|}\,a\cdot\{c_{\lambda}b\}.
\]
See, e.g., \cite{DZ01,Get02,LGCP}.
\end{rmk}

In Section~\ref{Sec:SSSPVA}, we adopt the sign conventions in Remark~\ref{rmk:CGPA}
(graded skew-symmetry and the odd Leibniz rule), which are standard in the integrable
systems literature; cf. \cite[(2.1.13)-(2.1.14)]{DZ01}.

\begin{exa}\label{exa:FLCSA}
Let $L:=\Cur(\mathfrak g)$ as in Example~\ref{exa:SCur}. Then $S^\bullet(L[-1])$ is a chiral Gerstenhaber algebra in the $\Z_2$-graded sense, where $[-1]$ denotes the parity shift. In particular, the odd Leibniz rule reads
\[
\{a_{\lambda}(b\cdot c)\}=\{a_\lambda b\}\cdot c+(-1)^{(|a|+1)|b|}\,b\cdot\{a_{\lambda}c\}.
\]
\end{exa}

Similar to Lemma 8 in \cite{K1} and to the corresponding proposition in \cite{BDSK09}, we have the following proposition.

Set $\widehat{\mathcal F}_{cl}^{\bullet}
:=\mathcal V^\bullet/\partial\mathcal V^\bullet,$
and denote by $\int:\mathcal V^\bullet\longrightarrow
\widehat{\mathcal F}_{cl}^{\bullet}$
the quotient map.
\begin{prop}\label{Prop:CGA}
Let $(\V^\bullet,\{-_{\lambda}-\},\partial,\cdot)$ be a chiral Gerstenhaber algebra. Then the following brackets are well-defined:
\begin{itemize}
\item[(1)] $\hFcl^\bullet\times\hFcl^\bullet\to\hFcl^\bullet,\quad
\{\int a,\int b\}:=\int\{a_{\lambda}b\}\big|_{\lambda=0}$,
\item[(2)] $\hFcl^\bullet\times\V^\bullet\to\V^\bullet,\quad
\{\int a,b\}:=\{a_{\lambda}b\}\big|_{\lambda=0}$.
\end{itemize}
Moreover, \textup{(1)} defines a graded Lie bracket on $\hFcl^\bullet$, and \textup{(2)} defines a representation of the graded Lie algebra $\hFcl^\bullet$ on $\V^\bullet$ by derivations whose degree is shifted by $1$; these derivations act on both the product and the $\lambda$-bracket of $\V^\bullet$ and commute with~$\partial$.
Here, “degree shifted by $1$” means that the bracket in \textup{(2)} satisfies
\[
\{\int a,bc\}=\{\int a,b\}\,c+(-1)^{|b|(|a|+1)}\,b\,\{\int a,c\}.
\]
\end{prop}

\begin{proof}
First, by left and right \emph{sesquilinearity}, the brackets in \textup{(1)} and \textup{(2)} are well-defined.

For all $\int a,\int b,\int c\in\hFcl^\bullet$, the graded skew-symmetry and the graded Jacobi identity of the chiral Gerstenhaber algebra imply that
\[
\{\int a,\int b\}=-(-1)^{(|a|+1)(|b|+1)}\{\int b,\int a\},
\]
and
\[
\{\int a,\{\int b,\int c\}\}=\{\{\int a,\int b\},\int c\}
+(-1)^{(|a|+1)(|b|+1)}\{\int b,\{\int a,\int c\}\},
\]
which shows that $\hFcl^\bullet$, together with the first bracket, is a graded Lie algebra.

We now show that the bracket in \textup{(2)} defines a representation of this graded Lie algebra.
By the above, \textup{(1)} defines a graded Lie algebra on $\hFcl^\bullet$.
Let
\[
\rho:\hFcl^\bullet\to\End(\V^\bullet),\qquad
\rho(\int a)(b):=\{\int a,b\}.
\]
By the shifted graded Jacobi identity~\eqref{eq:SGJ}, we obtain
\[
\rho(\{\int a,\int b\})=\rho(\int a)\circ\rho(\int b)
-(-1)^{(|a|+1)(|b|+1)}\,\rho(\int b)\circ\rho(\int a)\,,
\]
and
\[
\rho(\int a)(\{b_\mu c\})
=\{(\rho(\int a)b)_\mu c\}+(-1)^{(|a|+1)(|b|+1)}\{b_\mu\rho(\int a)c\},
\]
for all $a,b,c\in\V^\bullet$. Also, the shifted graded Leibniz rule implies that, for all $a,b,c\in\V^\bullet$,
$\{\int a,bc\}=\{\int a,b\}c+(-1)^{(|a|+1)|b|}b\{\int a,c\}.$
Finally, by sesquilinearity of $\V^\bullet$, we conclude the commutativity of $\partial$ with $\{\int a,-\}$ for all $a\in\V^\bullet$.
\end{proof}

\section{Shifted $1$-symplectic structure and Poisson vertex algebra}\label{Sec:SSSPVA}
In this section, we present our main theorem. 
We begin by recalling a theorem of Arakawa \cite{Ara12} (generalized to the super setting by Li~\cite{Li21}) on the Poisson vertex algebra structure on arc spaces. 
After that, we fix conventions for the geometric setup used later in this section.
\begin{thm}[{\cite{Ara12,Li21}}]\label{thm:level0}
Let \(X=\Spec R\) be an affine Poisson \emph{superscheme} (so \(R\) is a Poisson superalgebra).
Then the arc space \(J_\infty X\) carries a unique Poisson vertex superalgebra structure on
\[
R_\infty:=\OO(X_\infty)
\]
characterized by
\begin{equation}\label{eq:level0}
\{a_{\lambda} b\}=\{a,b\}, \qquad \text{for all } a,b\in R.
\end{equation}
\end{thm}

This implies the following proposition.
\begin{prop}\label{Prop:GeoCGA}
Let \(X=\Spec R\) be an affine scheme, where \(R\) is a Gerstenhaber algebra.
Then there exists a unique chiral Gerstenhaber algebra structure on
\[
R_\infty:=\OO(X_\infty)
\]
characterized by
\begin{equation}\label{eq:CGA}
\{a_\lambda b\}=\{a,b\},\qquad \text{for all } a,b\in R.
\end{equation}
\end{prop}
\begin{proof}
Since the translation operator \(\partial\) is even, the axioms of a chiral Gerstenhaber algebra follow directly from those of a level-$0$ Poisson vertex algebra.
\end{proof}
\subsection{Conventions}\label{Subsec:Conventions}
Now, we fix conventions. We mainly follow the notation of \cite[Chs.~1 and~4]{AraM24}.

We first fix notation.
Let \(X=\Spec R\) be a smooth affine scheme.
Denote its coordinate ring by \(\A_0:=\Gamma(X,\OO_X)\), and set \(Y:=T^*[1]X\) to be the degree-shifted cotangent bundle of \(X\).
Since $X$ is smooth, for any $x\in X$ there exists an affine open neighbourhood $U\ni x$
and a local coordinate system $\{u^i,\partial_i\}_{i=1}^n$ on $U$ in the sense of
\cite[Appendix~A.5]{HTT08}, i.e. $[\partial_i,\partial_j]=0$ and $\partial_i(u^j)=\delta_{ij}$.
In particular, $u=(u^1,\dots,u^n):U\to \BBA^n$ is \'{e}tale.
Let $\theta_1,\dots,\theta_n$ be the parity-reversed fiber coordinates dual to
$du^1,\dots,du^n$ (equivalently, corresponding to $\partial_1,\dots,\partial_n$).

Write \(\hatA_0:=\Gamma(Y,\OO_{Y})\) for the algebra of functions on \(Y\).
It has a natural decomposition
\[
\hatA_0=\bigoplus_{p=0}^n\hatA_0^p.
\]
\'Etale locally, \(\hatA_0^p\) consists of functions of the form
\[
P=\sum P^{\alpha_1\cdots\alpha_p}\,\theta_{\alpha_1}\theta_{\alpha_2}\cdots\theta_{\alpha_p},
\]
with the coefficients \(P^{\alpha_1\cdots\alpha_p}\) totally skew-symmetric in the indices \((\alpha_1,\dots,\alpha_p)\).

Passing to arc schemes, fix an affine open neighbourhood $U\subset X$
equipped with an étale coordinate map $u=(u^1,\dots,u^n):U\to\BBA^n$ as above.
We set
\[
\A_U:=\Gamma\bigl(U,(\pi_{\infty,0})_*\OO_{X_\infty}\bigr),\qquad
\hatA_U:=\Gamma\bigl(T^*[1]U,(\pi_{\infty,0})_*\OO_{Y_\infty}\bigr).
\]

We now give a geometric example of a chiral Gerstenhaber algebra.
\begin{exa}\label{Exa:GCGA}
Let \(X\) be an affine scheme over \(\C\), and set \(Y:=T^*[1]X\).
We use the same notation as in Example~\ref{exa:Tstar1X}.
By Proposition~\ref{Prop:GeoCGA}, on an \'etale open subset \(U\subset X\), the section \(\hatA_U\) carries a natural structure of a chiral Gerstenhaber algebra with \(\lambda\)-bracket
\[
\{\,u^\alpha{}_\lambda \theta_\beta\,\}=\delta^\alpha{}_{\beta},
\]
extended to the algebra by the (shifted) Leibniz rule.
\end{exa}

Therefore, following the notation of Section~\ref{Subsec:Conventions}, we obtain a chiral Gerstenhaber algebra
\[
\V^\bullet := \hatA_U
\]
with \(\lambda\)-bracket as in Example~\ref{Exa:GCGA}.
Let \(\partial\) denote the translation operator on \(\V^\bullet\).
We set
\[
\hFcl^\bullet \;:=\; \V^\bullet / \partial \V^\bullet
\]
and call \(\hFcl^\bullet\) the space of (classical) local functionals.
We write
\[
\int : \V^\bullet \twoheadrightarrow \hFcl^\bullet
\]
for the quotient map, and for \(P\in \V^\bullet\) we denote its class in \(\hFcl^\bullet\) by \(\int P\).
By Proposition~\ref{Prop:CGA}, the brackets from Proposition~\ref{Prop:CGA} endow \(\hFcl^\bullet\) with the structure of a graded Lie algebra,
\[
[-,-] : \hFcl^\bullet \otimes \hFcl^\bullet \longrightarrow \hFcl^\bullet,
\]
and make \(\V^\bullet\) a representation of \(\hFcl^\bullet\) by derivations.

\subsection{Normalizing operator and normalized expression}\label{Sec:NOE}
In the following, we work on an affine open neighbourhood $U$ equipped with an étale coordinate map $U\to\mathbb A^n$, and simply denote $\A_U$ and $\hatA_U$ by $\A$ and $\hatA$, respectively. We recall the normalizing operator \(\N\), first introduced by A.~Barakat in her Ph.D.\ thesis~\cite[Section~2.6]{Bar05}; see also~\cite[Lemma~2.3.6]{LZ11}.
For \(P=\int \tilde P\in\hFcl^{p}\) we set
\[
\N(P)\;:=\;\sum_{\alpha}\theta_{\alpha}\,\frac{\delta \tilde P}{\delta\theta_{\alpha}}.
\]
One checks that for every \(P=\int \tilde P\in\hFcl^{p}\),
\[
\int \N(P)\;=\;p\,P.
\]
For $p>0$, $\frac1p\mathcal N(P)$  provides a canonical normalized representative of each class in \(\hFcl^{p}\).

In particular, let
\[
P \;=\;\int\tilde{P}
\;=\;\int\sum_{s,t} P^{\alpha\beta}_{st}\,\theta_{\alpha}^{s}\theta_{\beta}^{t}
\;\in\;\hFcl^{2}.
\]
After applying \(\N\) and integrating by parts we obtain a normalized density of the form
\begin{equation}\label{Eq:nor}
\tilde P_{\mathrm{norm}}
\;=\;\frac12\sum_{\alpha,\beta,s}\theta_\alpha\bigl(\PD^{\alpha\beta}_s\partial^s\bigr)\theta_\beta.
\end{equation}
In the notation of integrable systems, this is encoded by the matrix differential operator
\[
\PD(\partial) \;=\; \bigl(\PD^{\alpha\beta}(\partial)\bigr)_{\alpha,\beta},
\qquad
\PD^{\alpha\beta}(\partial) \;=\; \sum_s \PD^{\alpha\beta}_s\,\partial^s.
\]

In the sequel we will always work with such normalized representatives, obtained from differential polynomials \(\tilde P\in\hat\A^2\).

As observed by Casati and Wang in~\cite[p.~1507]{CaW20} in the difference setting, and as explained to the author by M.~Casati (personal communication),
the graded commutativity of the odd variables \(\theta_\alpha\) and of the trace/integral implies that only the formally skew-adjoint part of $\mathcal P(\partial)$ contributes to the class \(P\).
We follow the notation of Section~\ref{Subsec:Conventions}; the main theorem below gives an explicit description of the graded Lie algebra structure on \(\hFcl^\bullet\).

\begin{thm}\label{Thm:NRB}
For any \(P\in\hatA^{p}\) and \(Q\in\hatA^{q}\), the graded Lie algebra structure on \(\hFcl^\bullet\), induced by the shifted symplectic structure on \(T^*[1]X\) and the shifted Leibniz rule (see Example~\ref{Exa:GCGA}), is given by
\[
\{\!\int P,\int Q \!\}
=\int (-1)^{p}\frac{\delta P}{\delta\theta_\sigma}\frac{\delta Q}{\delta u^\sigma}
+\frac{\delta P}{\delta u^\sigma}\frac{\delta Q}{\delta\theta_\sigma}.
\]
\end{thm}

\begin{proof}
In this proof, we adopt the following conventions: Greek indices are summed by the Einstein convention; \(m,n\) are summed explicitly.
Also, for differential monomials
\[
P=P^{\alpha_1\cdots\alpha_p}_{s_1\cdots s_p}\,\Theta^{s_1\cdots s_p}_{\alpha_1\cdots\alpha_p},
\qquad
Q=Q^{\beta_1\cdots\beta_q}_{r_1\cdots r_q}\,\Theta^{r_1\cdots r_q}_{\beta_1\cdots\beta_q},
\]
where \(\Theta^{s_1\cdots s_p}_{\alpha_1\cdots\alpha_p}=\theta_{\alpha_1}^{s_1}\theta_{\alpha_2}^{s_2}\cdots\theta_{\alpha_p}^{s_p}\) and \(\Theta^{r_1\cdots r_q}_{\beta_1\cdots\beta_q}=\theta_{\beta_1}^{r_1}\theta_{\beta_2}^{r_2}\cdots\theta_{\beta_q}^{r_q}\).
We have
\begin{align*}
\{P_\lambda Q\}
&=\bigl\{
P^{\alpha_1\cdots\alpha_p}_{s_1\cdots s_p}
\Theta^{s_1\cdots s_p}_{\alpha_1\cdots\alpha_p}{}_\lambda
Q^{\beta_1\cdots\beta_q}_{r_1\cdots r_q}
\bigr\}
\Theta^{r_1\cdots r_q}_{\beta_1\cdots\beta_q}\\
&\quad+Q^{\beta_1\cdots\beta_q}_{r_1\cdots r_q}
\bigl\{
P^{\alpha_1\cdots\alpha_p}_{s_1\cdots s_p}
\Theta^{s_1\cdots s_p}_{\alpha_1\cdots\alpha_p}{}_\lambda
\Theta^{r_1\cdots r_q}_{\beta_1\cdots\beta_q}
\bigr\}.
\end{align*}
Using the Master Formula and then setting \(\lambda=0\), we obtain
\begin{align*}
&\bigl\{
P^{\alpha_1\cdots\alpha_p}_{s_1\cdots s_p}
\Theta^{s_1\cdots s_p}_{\alpha_1\cdots\alpha_p}{}_\lambda
Q^{\beta_1\cdots\beta_q}_{r_1\cdots r_q}
\bigr\}
\Theta^{r_1\cdots r_q}_{\beta_1\cdots\beta_q}
\Big|_{\lambda=0}\\
&\qquad=(-1)^p\sum_{m,n\ge0}
\frac{\partial Q^{\beta_1\cdots\beta_q}_{r_1\cdots r_q}}{\partial u^{\sigma,m}}\,
\partial^m(-\partial)^n
\left(
P^{\alpha_1\cdots\alpha_p}_{s_1\cdots s_p}
\frac{\partial\Theta^{s_1\cdots s_p}_{\alpha_1\cdots\alpha_p}}{\partial\theta_\sigma^n}
\right)
\Theta^{r_1\cdots r_q}_{\beta_1\cdots\beta_q}.
\end{align*}
Similarly,
\begin{align*}
&Q^{\beta_1\cdots\beta_q}_{r_1\cdots r_q}
\bigl\{
P^{\alpha_1\cdots\alpha_p}_{s_1\cdots s_p}
\Theta^{s_1\cdots s_p}_{\alpha_1\cdots\alpha_p}{}_\lambda
\Theta^{r_1\cdots r_q}_{\beta_1\cdots\beta_q}
\bigr\}\Big|_{\lambda=0}\\
&\qquad=(-1)^{p(q-1)}Q^{\beta_1\cdots\beta_q}_{r_1\cdots r_q}
\sum_{m,n\ge0}
\frac{\partial\Theta^{r_1\cdots r_q}_{\beta_1\cdots\beta_q}}{\partial\theta_\sigma^n}\,
\partial^n(-\partial)^m
\left(
\frac{\partial P^{\alpha_1\cdots\alpha_p}_{s_1\cdots s_p}}{\partial u^{\sigma,m}}
\Theta^{s_1\cdots s_p}_{\alpha_1\cdots\alpha_p}
\right).
\end{align*}
Therefore, passing to the quotient, we find
\[
\int\{P_\lambda Q\}\Big|_{\lambda=0}
=\int\left(
(-1)^p\frac{\delta P}{\delta\theta_\sigma}\frac{\delta Q}{\delta u^\sigma}
+\frac{\delta P}{\delta u^\sigma}\frac{\delta Q}{\delta\theta_\sigma}
\right).
\]
Another way to prove it is to combine the uniqueness statement of \cite[Theorem~2.4.1]{LZ11} with Proposition~\ref{Prop:CGA}.
\end{proof}

\begin{rmk}\label{Rmk:NRB}
The same result extends to the setting of continuous Poisson vertex algebras (in the sense of De Sole-Kac-Valeri~\cite{DSKV24}) by working in the inverse-limit differential algebra and using continuous \(\lambda\)-brackets.
See also \cite{LZ11} for related infinite-dimensional Hamiltonian/Jacobi structures.
In addition, such Poisson vertex algebraic/bihamiltonian structures arise naturally in Gromov-Witten and cohomological field theories; compare, for example, \cite{DZ01,BDGR20,LWZ-series}.
\end{rmk}

Following \S4.1, let \(\mathscr H\) be the sheafification of the presheaf
\[
U \longmapsto \operatorname{coker}\bigl(\partial:\hatA_U\to \hatA_U\bigr).
\]
We call \(\mathscr H\) the sheaf of local functionals.
\begin{thm}[Correspondence Theorem]\label{Thm:CS}
Let \(X\) be a smooth, reduced, finitely generated \(\C\)-scheme, and set \(Y:=T^*[1]X\).
Then a degree-$2$ element \(P\in\Gamma(Y_\infty,\OO_{Y_\infty})_2\), through its class
\[
\int P\in\Gamma(X,\mathscr H)_2,
\]
defines a Poisson vertex algebra structure on \(X_\infty\) (i.e., it makes \(\OO_{X_\infty}\) a sheaf of PVAs) if and only if
\[
\bigl[\!\int P,\!\int P\bigr]=0,
\]
where \([\,\cdot\,,\,\cdot\,]\) denotes the Lie bracket on local functionals induced by the chiral Gerstenhaber bracket.\\
Conversely, any sheaf PVA \(\lambda\)-bracket on \(\OO_{X_\infty}\) determines a unique degree-\(2\) section of \(\mathscr H\),
which we denote by \(\int P\), such that
\[
\bigl[\!\int P,\!\int P\bigr]=0.
\]
\end{thm}

\begin{proof}
Since both the PVA axioms and the equation $\left[\int P,\int P\right]=0$ are local on $X$ and compatible with restriction, it suffices to work on an affine open cover by charts admitting étale coordinates. Fix such an affine open subset $U\subset X$ with étale coordinates $u=(u_1,\ldots,u_\ell):U\to\mathbb A^\ell$ as in Section~4.1, and write $\A:=\A_U$ and $\hatA:=\hatA_U$. The restriction of $\int P$ has a representative
\[
\widetilde P=\frac12\sum_{a,b=1}^{\ell}\theta_a H_{ab}(\partial)\theta_b\in\hatA^2,
\qquad H(\partial)^*=-H(\partial).
\]
For a degree-one expression $A=\sum_{m\geq0}A_m\theta_i^{(m)}$, write $\operatorname{ev}_{\lambda}(A)=\sum_{m\geq0}A_m\lambda^m$. For each $i,j$, define the antisymmetric symbol map $\sigma_{ji}^{\lambda,\mu}:\hatA^2\to\A[\lambda,\mu]$ by
\[
\sigma_{ji}^{\lambda,\mu}\bigl(f\theta_a^{(r)}\theta_b^{(s)}\bigr)
=f\bigl(\delta_{aj}\delta_{bi}\mu^r\lambda^s-\delta_{ai}\delta_{bj}\lambda^r\mu^s\bigr).
\]
The odd arc variables freely generate an exterior algebra over $\A$, so these maps respect anticommutation and are jointly injective. Moreover,
\[
\operatorname{ev}_{\lambda}(\partial A)=(\partial+\lambda)\operatorname{ev}_{\lambda}(A),
\qquad
\sigma_{ji}^{\lambda,\mu}(\partial B)=(\partial+\lambda+\mu)\sigma_{ji}^{\lambda,\mu}(B),
\]
where $\partial$ on the right acts on coefficients. Using the variational formula for the Schouten bracket, we get
\[
J:=\left[\int\widetilde P,\int\widetilde P\right]
=2\int\sum_{h=1}^{\ell}\frac{\delta\widetilde P}{\delta\theta_h}\frac{\delta\widetilde P}{\delta u_h}
=\int\widetilde J,
\]
where $\widetilde J=\sum_{h,r,a,b=1}^{\ell}\sum_{n\geq0}
\bigl(H_{hr}(\partial)\theta_r\bigr)(-\partial)^n
\left(\theta_a\frac{\partial H_{ab}(\partial)}{\partial u_h^{(n)}}\theta_b\right).$
Put $E_k:=\delta\widetilde J/\delta\theta_k$. The normalizing identity gives $3J=\int\N(\int\widetilde J)=\int\sum_k\theta_kE_k$. Since $J=0$ in $\cH_3(U)$ means that $\widetilde J$ is locally a total derivative, and variational derivatives annihilate total derivatives, we obtain $J=0$ if and only if $E_k=0$ for all $k$. Taking the odd variational derivative of $\widetilde J$ and using $H(\partial)^*=-H(\partial)$ gives
\begin{equation}\label{Eq:NJCPT}
\begin{aligned}
E_k=-\sum_{j,i,h=1}^{\ell}\sum_{n\geq0}\Biggl[
&H_{kh}(\partial)(-\partial)^n
\left(\theta_j\frac{\partial H_{ji}(\partial)}{\partial u_h^{(n)}}\theta_i\right)\\
&-\left(\frac{\partial H_{kj}(\partial)}{\partial u_h^{(n)}}\theta_j\right)
\partial^n\bigl(H_{hi}(\partial)\theta_i\bigr)\\
&+\partial^n\bigl(H_{hj}(\partial)\theta_j\bigr)
\left(\frac{\partial H_{ki}(\partial)}{\partial u_h^{(n)}}\theta_i\right)
\Biggr].
\end{aligned}
\end{equation}
Here the sum over $i,j$ is taken before extracting any symbol. The first term comes from varying $H_{hr}(\partial)\theta_r$; the other two come from varying the two odd factors in parentheses. Set $\nu=\lambda+\mu$. Writing $H_{ji}(\partial)=\sum_r H_{ji;r}\partial^r$, skew-adjointness gives $H_{ij}(\mu)=-\sum_r(-\mu-\partial)^rH_{ji;r}$. Together with \cite[(1.2)]{BDSK09}, this yields
\[
\sum_{n\geq0}(-\nu-\partial)^n\frac{\partial H_{ij}(\mu)}{\partial u_h^{(n)}}
=-\sum_{n\geq0}(-\nu-\partial)^n\frac{\partial H_{ji}(\lambda)}{\partial u_h^{(n)}}.
\]
Consequently, applying $\sigma_{ji}^{\lambda,\mu}$ to \eqref{Eq:NJCPT} gives
\[
\begin{aligned}
\frac12\sigma_{ji}^{\lambda,\mu}(E_k)
=\sum_{h=1}^{\ell}\sum_{n\geq0}\Biggl(
&\frac{\partial H_{kj}(\mu)}{\partial u_h^{(n)}}(\lambda+\partial)^nH_{hi}(\lambda)\\
&-\frac{\partial H_{ki}(\lambda)}{\partial u_h^{(n)}}(\mu+\partial)^nH_{hj}(\mu)\\
&-H_{kh}(\lambda+\mu+\partial)(-\lambda-\mu-\partial)^n
\frac{\partial H_{ji}(\lambda)}{\partial u_h^{(n)}}
\Biggr).
\end{aligned}
\]
All occurrences of $\partial$ act on the coefficients to their right. The expression on the right is the left-hand side minus the right-hand side of equation~(1.49) in \cite{BDSK09}. By Proposition~1.16(b) of \cite{BDSK09}, it vanishes for every $i,j,k$ if and only if the bracket defined by the master formula from $\{u_i{}_\lambda u_j\}=H_{ji}(\lambda)$
satisfies the Jacobi identity. Joint injectivity of the symbol maps and the normalizing identity therefore give $\left[\int P,\int P\right]=0$ if and only if $\{\cdot{}_\lambda\cdot\}_H$ is a PVA bracket.
Conversely, a sheaf PVA bracket determines the local operators $H_{ji}(\lambda)=\{u_i{}_\lambda u_j\}$. Their skew-adjoint normal forms determine unique degree-two local-functional classes. By the chain rule, the local correspondence is compatible with changes of coordinates and restriction, so the brackets in the forward construction and the classes in the converse construction glue, yielding the asserted section of $\cH_2$.
\end{proof}

\begin{prop}
Let $P\in\hFcl^2$ be written in the normalized
form of Section~4.2,
\[
P=\frac12\int\sum_s
\theta_\alpha P_s^{\alpha\beta}\theta_\beta^s,
\]
with $[P,P]=0.$
For \(F,G\in\hFcl^{0}\) define the derived bracket by
\[
\{F,G\}_{P}:=\{\{P,F\},G\}.
\]
Then
\[
\{F,G\}_{P}
=\int \frac{\delta G}{\delta u^\beta}\,\PD^{\beta\alpha}(\partial)\,\frac{\delta F}{\delta u^\alpha}.
\]
Moreover, the corresponding PVA structure is given by $\{u^\alpha{}_\lambda u^\beta\}=\PD^{\beta\alpha}(\lambda)$.
\end{prop}

\begin{proof}
For $F\in\hFcl^0$, write $\{P,F\}=\int X^\alpha\theta_\alpha$.
By Theorem~\ref{Thm:NRB} and the formal skew-adjointness of $\PD(\partial)$,
\[
X^\alpha=-\PD^{\alpha\beta}(\partial)\,\frac{\delta F}{\delta u^\beta}.
\]
Hence
\[
\begin{aligned}
\{F,G\}_P
&=\{\{P,F\},G\}
=\left\{\int X^\alpha\theta_\alpha,G\right\}\\
&=-\int\frac{\delta G}{\delta u^\alpha}X^\alpha
=\int\frac{\delta G}{\delta u^\alpha}\PD^{\alpha\beta}(\partial)\frac{\delta F}{\delta u^\beta}.
\end{aligned}
\]
Relabelling $\alpha$ and $\beta$ gives the stated formula. The final assertion follows from Theorem~\ref{Thm:CS}.
\end{proof}
We will give a more geometric interpretation of this statement in future work.\par\medskip
We now illustrate how the above shifted symplectic formalism reproduces some standard Poisson vertex algebras.

\begin{exa}[Heisenberg PVA]\label{exa:Hei}
Consider the Heisenberg Poisson vertex algebra.
Geometrically, we take \(X=\BBA^1\) with coordinate \(u\) and its arc space \(X_\infty\), so that
\[
\A := \C[J_\infty X]
= \C\bigl[u^{(n)}\mid n\ge0\bigr],
\]
and the PVA structure is given by
\[
\{u_{\lambda}u\}=\lambda.
\]
In the shifted cotangent picture, we consider the odd cotangent bundle \(Y=T^*[1]X\) with odd fiber coordinate \(\theta\), and its arc space \(Y_\infty\) with jet coordinates \(\theta^{(n)}\).
The above \(\lambda\)-bracket corresponds to the section
\[
P = \frac{1}{2}\theta\,\theta^{(1)}\ \in \ \Gamma\bigl(Y_\infty,\OO_{Y_\infty}\bigr)_2,
\]
viewed as a degree-2 function on the arc space $Y_\infty$.
A direct computation gives
\[
[P_\lambda P]\big|_{\lambda=0}=0,
\]
so \(P\) defines a Hamiltonian operator on \(X_\infty\).
Equivalently, the associated Hamiltonian operator is the constant operator \(\PD(\partial)=\partial\), and the derived bracket recovers the Heisenberg PVA above.
\end{exa}

\begin{exa}[Virasoro-Magri PVA]\label{exa:Vir}
Let \(X=\BBA^1\) with coordinate \(u\), and consider its arc space \(X_\infty\) and odd cotangent bundle \(Y=T^*[1]X\) as above.
On \(\A=\C[J_\infty X]\) we consider the Virasoro-Magri PVA structure
\[
\{u_{\lambda}u\}=u'+2u\,\lambda+c\,\lambda^3,
\]
where \(c\in\C\) is the central charge.
The corresponding Hamiltonian operator is
\[
\PD(\partial)=u'+2u\,\partial+c\,\partial^3.
\]
Using Theorem~\ref{Thm:CS}, it corresponds to a degree-$2$ section
\[
P=\frac{1}{2}(2u\,\theta\,\theta^{(1)}+c\,\theta\,\theta^{(3)})
\]
in \(\Gamma(Y_\infty,\OO_{Y_\infty})\).
Using the definition of the \(\lambda\)-bracket, one computes
\[
[P_\lambda P]\big|_{\lambda=0}
=\frac c2\,\partial\bigl(\theta\theta^{(1)}\theta^{(2)}\bigr)
\in\partial\hatA.
\]
Therefore, \(P\) defines a Hamiltonian operator on \(X_\infty\), and the associated derived bracket induces the Virasoro-Magri PVA structure above.
\end{exa}

\section{Applications: PVA sheaves and the classical R-matrix}\label{Sec:App}
The main goal of this section is to present several applications of Theorem~\ref{Thm:CS}. 
In Section~\ref{Sec:SPVA} we give several examples of sheaves of Poisson vertex algebras on different algebraic schemes; in particular, we obtain a classification of certain PVA sheaves on some algebraic curves. 
In Section~\ref{subsec:CR} we give a Maurer-Cartan description of the classical $R$-matrix of Lie conformal algebras, as introduced by the author in~\cite{Fang25}.

\subsection{Sheaves of PVAs}\label{Sec:SPVA}
In this section we apply Theorem~\ref{Thm:CS} to classify global scalar PVA structures on the arc space $X_\infty$ for $X=\mathbb{CP}^1$.

We restrict our attention to $\lambda$-brackets of order at most~$3$, since this class contains the Dubrovin-Novikov hydrodynamic-type Hamiltonian operators~\cite{DN83} and it is the natural conformal (CFT-type) case in the sense of De~Sole-Kac (see, e.g.,~\cite{DSKW10}).

By Theorem~4.6, a global scalar PVA structure corresponds to a
degree-$2$ section
\[
\int P\in \Gamma(X,\mathcal H)_2
\]
satisfying
\[
\left[\int P,\int P\right]=0.
\]
On the standard affine cover of $\mathbb P^1$, we first determine the
local degree-$2$ representatives which are compatible on the overlap,
and then impose the Maurer--Cartan equation.
\begin{exa}\label{exa:CP1}
Let $X=\CP^1$ with the standard affine cover $U_i$ ($i=1,2$). 
We use a local coordinate $u$ on $U_1$ and $v$ on $U_2$, related on $U_1\cap U_2$ by
\[
v=\frac{1}{u}.
\]
We restrict to the branch with vanishing $\lambda^2$-coefficient.
Then skew-symmetry implies that the leading coefficient $a_3$
satisfies $\partial a_3=0$. Since $\ker\partial=\mathbb C$, we write
$a_3=c/12$, and hence the bracket takes the form
\begin{equation}\label{eq:CP1U1}
\{u_\lambda u\}
=\frac{1}{2} (f_1)' + f_1\,\lambda + \frac{c}{12}\lambda^3,
\qquad
f_1\in \A_{U_1},\ c\in\C.
\end{equation}

The coordinate change $u\mapsto v=\tfrac1u$ induces a transformation on the fiber coordinates; with the convention $\theta_u=\partial/\partial u$, $\theta_v=\partial/\partial v$, we have on $U_1\cap U_2$
\[
\theta_u = -\frac{1}{u^2}\,\theta_v.
\]
For brevity we write $\theta:=\theta_u$; the $\lambda$-bracket~\eqref{eq:CP1U1} corresponds to a Hamiltonian function
\[
P \;=\;\frac{1}{2}( f_1\,\theta\,\theta^{(1)}+\frac{c}{12}\,\theta\,\theta^{(3)}
\;)\in\;
\Gamma\bigl(T^*[1]U_1,\pi_*\OO_{Y_\infty}\bigr)_2.
\]

Choosing local degree-$2$ representatives on the standard affine cover
and imposing the compatibility condition on $U_1\cap U_2$, together
with the PVA Jacobi identity on the possible cancellation terms, we find
that $f_1$ lies in the span of the following differential monomials:
\begin{equation}\label{eq:GSCP1}
1,\ u,\ u^2,\ u^{(1)},\ u\,u^{(1)},\ u^3,\ u^{(2)},\ u\,u^{(2)},\ u^2 u^{(1)},\ u^4,\ (u^{(1)})^2,\ u^{(3)}.
\end{equation}
Within this finite-dimensional ansatz, the remaining Jacobi equations
give the following classification of global scalar PVA brackets of the
form~\eqref{eq:CP1U1}.

\smallskip
\noindent
(a) If $c\neq 0$, the PVA Jacobi identity implies
\[
f_1=a_0+a_1u,\qquad a_0,a_1\in\mathbb C.
\]
If $a_1\neq0$, then, up to an affine change of the coordinate $u$,
the bracket is of Virasoro--Magri type
\[
\{u_\lambda u\}
=
u'+2u\lambda+\frac{\tilde c}{12}\lambda^3
\]
for some $\tilde c\neq0$.
If $a_1=0$, one obtains the constant-coefficient family
\[
\{u_\lambda u\}
=
a_0\lambda+\frac c{12}\lambda^3.
\]
\smallskip
\noindent
(b) If \(c=0\), then all global scalar PVA brackets on \(X_\infty\) of the form~\eqref{eq:CP1U1} are given by
\[
\{u_\lambda u\}= q'(u)\,u' + 2q(u)\,\lambda,
\]
where, on an affine chart, $q(u)$ is a polynomial of degree at most~$4$, equivalently
\[
q\in H^0\bigl(X,\OO_X(4)\bigr).
\]
The bound $\deg q\leq4$ reflects regularity at $p=\infty$ after the coordinate change $v=\tfrac1u$ and the fiber transformation $\theta_u=-u^{-2}\theta_v.$
\end{exa}
\begin{rmk}\label{rmk:SIE}
We denote $p=\infty$ and note that
\[
\ord_p(u)=-1, \qquad \ord_p\bigl(u^{(k)}\bigr)\ge -1-k.
\]
Write $f=\sum_i c_i\,f_i$ as a sum of differential monomials
\[
f_i=\prod_k\bigl(u^{(k)}\bigr)^{a_k^i}.
\]
Then we obtain the estimate
\begin{equation}\label{eq:SIE}
\ord_p\!\left(\prod_k\bigl(u^{(k)}\bigr)^{a_k^i}\right)
\;\ge\;\sum_k(-k-1)a_k^i.
\end{equation}
We call inequality~\eqref{eq:SIE} the sieve inequality. 

Equation~\eqref{eq:SIE} gives \emph{termwise sufficient} conditions on the exponents $a_k^i$
to guarantee regularity at $p$ after multiplying by the fiber factor.
In particular, in the Example~\ref{exa:CP1} we get the sufficient bounds
$\sum_{k}(k+1)a_k^1\leq4,\ \sum_k(k+1)a_k^3\leq2.$
\end{rmk}
\subsection{Classical R-matrix}\label{subsec:CR}
Recall in the paper \cite{Fang25}, the author introduced the classical \(R\)-matrix of Lie conformal algebra. Theorem \ref{Thm:CS} implies the following corollary.

\begin{thm}\label{Thm:RM}
Let $(L,\partial,[\cdot_\lambda\cdot])$ be a Lie conformal algebra, either free of finite rank over $\mathbb C[\partial]$ or affine $L=v_k(\mathfrak g)$ with $\mathfrak g$ finite-dimensional. Set $\mathcal V:=S^\bullet(L)$ in the former case and $\mathcal V:=\operatorname{Sym}'(L)$ in the latter, endowed with the induced PVA structure as in \cite[Remark~2.4]{Fang25}.

Set \(N:=\mathrm{rank}_{\C[\partial]}(L)\) and let \(X:=\BBA^{N}\).
Let \(X_\infty\) be the arc space of \(X\). Then \(\OO(X_\infty)\cong \V\) as
differential algebras; hence we identify \(X_\infty\simeq\Spec(\V)\).
Set \(Y:=T^*[1]X\).
Suppose \(R\in\End_{\C}(L)\) satisfies \([R,\partial]=0\).

Then \(R\) is a classical $R$-matrix if and only if the element
\[
P_R \;=\; \frac{1}{2}\theta_\alpha\,\{\,u^{\beta}{}_{\partial}u^{\alpha}\,\}_{R,\rightarrow}\,\theta_\beta
\ \in\ \hatA^{2}
\]
satisfies \(\{\!\int P_R,\int P_R\!\}=0\). (Here \({}_{\rightarrow}\) means: replace
\(\lambda\) by \(\partial\) and move all \(\partial\)’s to the right.)
Equivalently, set $\Theta_R :=\int(P_R - P)$. Then $\Theta_R$ is a Maurer-Cartan element of the
dg Lie algebra $(\hFcl, [ - , - ], d_P:=[\int P,-])$, i.e.\ $d_P\Theta_R+\frac12[\Theta_R,\Theta_R]=0.$
\end{thm}

\begin{exa}
In the affine Kac--Moody Lie conformal algebra case, let $\V:=\V_k(\mathfrak{g})=\operatorname{Sym}'(v_k(\mathfrak{g}))$ as in \cite[Remark~2.4]{Fang25}.
We identify $\Spec\V$ with $\mathfrak{g}^*_{k,\infty}$, where $k$ indicates the level of the $\lambda$-bracket. In particular, \cite[Corollary~5.6]{Fang25} shows that if $(\mathfrak{g},(\,\cdot\,,\,\cdot\,))$ is a quadratic Lie algebra and $(\mathfrak{g},\mathfrak{g}_+,\mathfrak{g}_-)$ is a Manin triple with respect to $(\,\cdot\,,\,\cdot\,)$, then the projection maps give a classical $R$-matrix for $v_k(\mathfrak{g})$.
This can be viewed as a chiral counterpart of the classical fact that a Manin triple gives a projection-type classical $R$-matrix.
\end{exa}
Motivated by the above discussion, we now define the notion of a classical \(R\)-matrix for a classical finite \(\W\)-algebra.

\begin{dfn}\label{dfn:CFWRM}
Let \(\W_{\fin}^{\cl}(\fg,f_{\sub})\) be the subregular classical finite \(\W\)-algebra.
As a Poisson variety one has
\[
\Spec \W_{\fin}^{\cl}(\fg,f_{\sub}) \;\cong\; \Sf,
\]
where \(\Sf\) is the Slodowy slice associated to \(f_{\sub}\), equipped with the
Poisson structure induced from the Kirillov-Kostant Poisson structure on \(\fg^*\).
We denote by \(\pi_f\in\Gamma(\Sf,\wedge^2 T\Sf)\) the corresponding Poisson bivector.

Set \(X := \Sf\) and \(Y := T^*[1]X\).
We say that an element \(\Pi_R\in\Gamma(Y,\OO_Y)_2\) is a \emph{classical \(R\)-matrix}
of \(\W^{\cl}_{\fin}(\fg,f_{\mathrm{sub}})\) if and only if
$d_{\pi_f}\Theta_R+\tfrac12[\Theta_R,\Theta_R]=0,$ where $\Theta_R:=\Pi_R-\pi_f$ and \([\cdot,\cdot]\) is the Schouten–Nijenhuis bracket on polyvector fields.
\end{dfn}
\begin{exa}\label{exa:CDSsl3}
Let $\fg=\fsl_3$, and let $c$ be a Poisson-central polynomial
coordinate of $\W^{\cl}_{\fin}(\fsl_3,f_{\sub})$.
Following the reduction to a $3\times3$ Poisson matrix of
Damianou--Sabourin--Vanhaecke, we specialize $c$ to zero and set
\[
A_0:=\W^{\cl}_{\fin}(\fsl_3,f_{\sub})/(c)
\cong\C[u_1,u_2,u_3].
\]
The induced Poisson structure on $A_0$ is given by
\[
H=
\begin{bmatrix}
0 & -2u_1 & u_2^2 \\
2u_1 & 0 & -2u_3 \\
-u_2^2 & 2u_3 & 0
\end{bmatrix}.
\]Let $X=\Spec A_0$ and $Y=T^*[1]X$. Set
$V=\Span_\C\{u_1,u_2,u_3\}$, and let
$R\in\End_\C(V)$ be given by
$R(u_i)=\sum_{j=1}^3 R_{ij}u_j$, where $R_{ij}\in\C$.
We will show in~\cite{Fang26RDW0} that, whenever $R$ satisfies
the Jacobi relations below, extending it by $R(c)=0$ yields a
vertical $R$-matrix on the full Slodowy slice $\Sf$. The
deformation on $X$ is its specialization at $c=0$.
It induces a degree-$2$ function
\[
\Pi_R=-\frac{1}{2}\sum_{i,j=1}^3
\bigl(\{R(u_i),u_j\}+\{u_i,R(u_j)\}\bigr)\theta_i\theta_j
\in\Gamma(Y,\OO_Y)_2,
\]
where $\theta_i$ are the odd coordinates corresponding to $u_i$.
The Jacobi identity for $\Pi_R$ is equivalent to the polynomial
relations
\begin{align*}
f_1&:=R_{11}R_{21}+R_{21}R_{22}+R_{23}R_{31}=0,\\
f_2&:=R_{12}R_{31}+R_{22}R_{32}+R_{32}R_{33}=0,\\
f_3&:=R_{11}^2+R_{12}R_{21}-R_{23}R_{32}-R_{33}^2=0,\\
f_4&:=R_{13}R_{21}+R_{22}R_{23}+R_{23}R_{33}=0,\\
f_5&:=R_{11}R_{12}+R_{12}R_{22}+R_{13}R_{32}=0.
\end{align*}
We call $\Pi_R$ a constant $R$-deformation of the
Poisson algebra $A_0$. In algebraic–geometric terms, put
$\C[R]=\C[R_{ij}\mid 1\le i,j\le 3]$ so that $\Spec\C[R]\simeq\BBA^9$
and let $I=\langle f_1,\dots,f_5\rangle\subset\C[R]$. Then $\mathfrak R:=V(I)\subset\mathbb A^9$ parametrizes
the constant $3\times3$ matrices $R$ for which $\Pi_R$
is Poisson on $X$. For each such $R$, let
$\widetilde\Pi_R$ denote the Poisson bivector on $\Sf$
obtained by the vertical extension $R(c)=0$.
It is $\widetilde\Pi_R$ that is a classical $R$-matrix
in the sense of Definition~\ref{dfn:CFWRM}.
More general classical $R$-matrices on $\Sf$ satisfy
the Maurer--Cartan equation of Definition~\ref{dfn:CFWRM}.
Via the vertical extension $R(c)=0$, the variety
$\mathfrak R$ parametrizes a family of constant solutions
to that equation.
\\
In analogy with the classical study of constant solutions of the CYBE (see, e.g.,~\cite{BD82}), we single out three natural families of constant endomorphisms $R$.\\
\medskip\noindent
\emph{Type I (Cartan type).}
Assume $R$ is diagonal. Then $R$ defines a constant $R$–deformation if
and only if $R=\mathrm{diag}(a,b,\pm a),\ a,b\in\C.$

\medskip\noindent
\emph{Type II (upper triangular type).}
Assume $R$ is upper triangular, i.e.\ $R_{21}=R_{31}=R_{32}=0$. Then the
Jacobi identity for $\Pi_R$ is equivalent to
\[
R_{11}^2=R_{33}^2,\qquad
R_{12}(R_{11}+R_{22})=0,\qquad
R_{23}(R_{22}+R_{33})=0.
\]

\medskip\noindent
\emph{Type III (mixed type).}
Assume $R_{12}=R_{21}=R_{23}=R_{32}=0.$
Then $R$ defines a constant $R$–deformation if and only if $R_{11}=\pm R_{33}.$\\
From now on we choose the $+$ branch $R_{11}=R_{33}$ in Type~III. We observe that the Adler-Kostant-Symes scheme remains valid under
constant $R$-deformations. Let
\(Z_{\Pois}(A_0)\)
denote the Poisson center of \(A_0\), and write
\(A_0{,_R}\) for the $R$-deformed $A_0$. Then, for any $f,g\in Z_{\mathrm{Poi}}(A_0)$, we have
$\{f,g\}_R=0$, where $\{\cdot,\cdot\}_R$ denotes the
$R$-deformed Poisson bracket. Consequently, the elements of
\(Z_{\Pois}(A_0)\) provide a family
of mutually commuting Hamiltonians, and by restricting to suitable
symplectic leaves we obtain integrable systems. The proof is entirely
parallel to that of~\cite[Theorem~6.3]{Fang25}. We now consider the mixed-type constant $R$-matrix of Type III.
The polynomial
\[
S:=\frac16u_2^3+u_1u_3
\]
is a Casimir of the original Poisson structure on $A_0$. Taking $S$ as the Hamiltonian for the $R$-deformed Poisson structure on $A_0$, with the convention $\dot u_i=\{u_i,S\}_R$, gives the system
\begin{equation}\label{eq:EUL}
\begin{aligned}
\dot{u}_1 &= u_2^2(\Delta\, u_1 + R_{13}u_3),\\
\dot{u}_2 &= 2\big(R_{31}u_1^2 - R_{13}u_3^2\big),\\
\dot{u}_3 &= -\,u_2^2\big(R_{31}u_1 + \Delta\, u_3\big),
\end{aligned}
\end{equation}
where $\Delta := R_{33}-R_{22}$.

The chosen Type III $R$-matrix also produces a compatible
Poisson pencil on $X=\Spec A_0$. More precisely, let $H$
denote the original Poisson matrix on $X$, and let
$H^{\mathrm{mix}}$ denote the Poisson matrix induced by this
mixed-type $R$-matrix. Then
$H+\epsilon H^{\mathrm{mix}}$ is a Poisson matrix on $X$
for every $\epsilon\in\C$.

We now specialize the parameters $R_{ij}$ to obtain a
genus-two hyperelliptic reduction. For convenience we impose $\Delta = 0$, assume $R_{13}\neq 0$ and
$R_{31}\neq 0$, and rescale the time variable and $u_i$ so that
$R_{13}=R_{31}=1$. Then \eqref{eq:EUL} reduces to
\[
\dot{u}_1 = u_3 u_2^2,\qquad
\dot{u}_2 = 2(u_1^2 - u_3^2),\qquad
\dot{u}_3 = -u_1 u_2^2.
\]
The system also preserves
\[
J:=u_1^2+u_3^2.
\]
Since $u_1u_3=S-\frac16u_2^3$, we obtain
\[
\dot u_2^{\,2}=P(u_2),\qquad
P(u_2):=4J^2-16\left(S-\frac16u_2^3\right)^2
=\frac49\left(u_2^3-3(2S-J)\right)
       \left(3(2S+J)-u_2^3\right).
\]
For fixed values $S=s$ and $J=j$ satisfying
$j(j^2-4s^2)\neq0$, the polynomial $P(x)$ has six distinct
roots. Hence the smooth projective model of $y^2=P(x)$ is
a genus-two hyperelliptic curve.
\end{exa}
We summarize the above example as follows.
\begin{prop}\label{prop:R-types}
With notation as above, the following three families of
endomorphisms $R\in\End_\C(V)$ define constant
$R$-deformations of the Poisson algebra $A_0$. Their
extensions by $R(c)=0$ give vertical classical $R$-matrices
on the full Slodowy slice $\Sf$.
\begin{description}
  \item[Type I (Cartan type)]
  Assume $R$ is diagonal. Then $R$ defines a constant $R$–deformation
  if and only if
  \[
    R=\mathrm{diag}(a,b,\pm a),\qquad a,b\in\C.
  \]

  \item[Type II (upper triangular type)]
  Assume $R$ is upper triangular, i.e.\ $R_{21}=R_{31}=R_{32}=0$.
  Then the Jacobi identity for $\Pi_R$ is equivalent to
  \[
    R_{11}^2=R_{33}^2,\qquad
    R_{12}(R_{11}+R_{22})=0,\qquad
    R_{23}(R_{22}+R_{33})=0.
  \]

  \item[Type III (mixed type)]
  Assume $R_{12}=R_{21}=R_{23}=R_{32}=0$. Then $R$ defines a constant
  $R$–deformation if and only if $R_{11}=\pm R_{33}$.
\end{description}
In particular, any Type~III$_+$ $R$-matrix (i.e.\ with $R_{11}=R_{33}$) produces a compatible Poisson pencil.
\end{prop}

\begin{thm}\label{thm:RDWSY}
In the setting of the Poisson algebra $A_0$ with Poisson
matrix $H$ as above, consider a constant $R$-matrix of
Type III$_+$ and impose $\Delta:=R_{33}-R_{22}=0$.
Assume $R_{13}\neq 0$ and $R_{31}\neq 0$, and rescale the time
variable and $u_i$ so that $R_{13}=R_{31}=1$. Then the Hamiltonian flow generated by
\[
S=\frac{1}{6}u_2^3+u_1u_3
\]
with respect to the $R$-deformed Poisson bracket on $A_0$
is given by the three-dimensional system
\begin{equation}\label{eq:RDW}
\dot{u}_1 = u_3 u_2^2,\qquad
\dot{u}_2 = 2(u_1^2 - u_3^2),\qquad
\dot{u}_3 = -u_1 u_2^2.
\end{equation}
This system has another integral of motion, $J=u_1^2+u_3^2$. Moreover, for generic values of $S$ and $J$, the component $u_2(t)$ can be expressed in terms of hyperelliptic functions.
\end{thm}
We call the system above the \emph{centrally reduced
RDW system at $c=0$}, or simply the \emph{centrally
reduced RDW system (CRDW system)}.
Here the central reduction refers specifically to
setting the Poisson-central coordinate $c$ equal to zero,
yielding the quotient algebra $A_0$.
The following proposition gives a Lax representation
for its hyperelliptic reduction.

\begin{prop}\label{Prop:LAX}
Consider the $R$-deformed system~\eqref{eq:RDW} in the mixed-type
case with conserved quantities
\(
S=\tfrac16 u_2^3+u_1u_3
\)
and
\(
J=u_1^2+u_3^2.
\)
Set \(x:=u_2\) and \(y:=3(u_1^2-u_3^2)\), so that
\[
y^2=g(x),\qquad
g(\lambda):=-\lambda^6+12S\lambda^3+9\bigl(J^2-4S^2\bigr).
\]
Fix a common level set $S=s$, $J=j$, and choose
$b\in\C$ with $b^2=-1+12s+9(j^2-4s^2)$.
On this level set, work on the Zariski open subset
\[
U:=\{x\neq 1,\ y\neq 0\}.
\]
Define on \(U\)
\[
u(\lambda)=(\lambda-x)(\lambda-1),\qquad
m:=\frac{y-b}{x-1},\qquad
v(\lambda)=b+m(\lambda-1),
\]
and set
\[
w(\lambda):=\frac{g(\lambda)-v(\lambda)^2}{u(\lambda)}\in \mathcal O(U)[\lambda].
\]
(Indeed, since \(g(x)=y^2\) and \(g(1)=b^2\), the polynomial \(g(\lambda)-v(\lambda)^2\)
vanishes at \(\lambda=x\) and \(\lambda=1\), hence it is divisible by \(u(\lambda)\).)
Let
\[
L(\lambda)=\begin{bmatrix} v(\lambda) & w(\lambda)\\ u(\lambda) & -v(\lambda)\end{bmatrix},
\qquad
M(\lambda)=-\frac{1}{\lambda-x}\begin{bmatrix}
\tfrac13 y & \Phi(x,y)\\[2pt]
0 & -\tfrac13 y
\end{bmatrix},
\]
where
\[
\Phi(x,y):=\frac{(x-1)g'(x)-2y(y-b)}{3(x-1)^2}\in \mathcal O(U).
\]
Then, on $U$, the system~\eqref{eq:RDW} admits the Lax
representation
\[
\dot L(\lambda)=[L(\lambda),M(\lambda)]
\qquad
\text{in }\operatorname{Mat}_2(\mathcal O(U)(\lambda)).
\]
In particular, this Lax representation is of even Mumford type.
\end{prop}

\begin{cor}\label{cor:GCM}
Let $\nabla=d-L(\lambda)\,d\lambda$ be the meromorphic connection associated to~\eqref{eq:RDW}.
Then the Boalch--Yamakawa core diagram of the formal class of $\nabla$ at $\infty$
consists of two nodes $\langle q_\pm\rangle$ joined by three edges.
Equivalently, the associated generalized Cartan matrix is
\[
C=\begin{bmatrix}2&-3\\-3&2\end{bmatrix}.
\]
\end{cor}

\begin{proof}
Since $\mathrm{tr}\,L(\lambda)=0$ and $\det(\eta I-L(\lambda))=\eta^2-g(\lambda)$,
the eigenvalues of $L(\lambda)\,d\lambda$ are $\pm\xi(\lambda)\,d\lambda$ with
$\xi(\lambda)=\sqrt{g(\lambda)}$.
Choose a branch near $\lambda=\infty$. Here the exponential polynomials $q_\pm$ are the polynomial parts, with zero constant term, of $\pm\int^\lambda\xi(t)\,dt$. Since $\xi(\lambda)=c\lambda^3+O(1)$ with $c\ne0$, we have $q_+-q_-=\frac c2\lambda^4+O(\lambda)$.
Thus $\deg(q_+-q_-)=4$ and the factors are unramified ($\beta_\pm=1$).
By the untwisted Boalch--Yamakawa rule, the number of edges equals $\deg(q_+-q_-)-1=3$.
\end{proof}

\begin{rmk}\label{rmk:RDWQV}
By Corollary~\ref{cor:GCM}, the connection $\nabla=d-L(\lambda)\,d\lambda$
has the following core diagram:
\begin{center}
\begin{tikzpicture}[
  vertex/.style={draw,circle,minimum size=11mm,inner sep=1pt},
  line/.style={line width=0.6pt}
]
  \node[vertex] (v1) at (0,0) {$\langle q_+\rangle$};
  \node[vertex] (v2) at (3.2,0) {$\langle q_-\rangle$};
  \draw[line] (v1) to[bend left=18]  (v2);
  \draw[line] (v1) --               (v2);
  \draw[line] (v1) to[bend right=18] (v2);
\end{tikzpicture}
\end{center}
The associated Cartan matrix $C=\begin{bmatrix}2&-3\\-3&2\end{bmatrix}$
is indefinite (indeed $\det C=-5<0$; in rank two this is hyperbolic).

Moreover, let $\Theta=\langle q_+\rangle+\langle q_-\rangle$
be the irregular class of $\nabla$ at $\infty$, and fix
formal monodromy conjugacy classes
$\mathcal C=(\mathcal C_+,\mathcal C_-)$. If the corresponding
stable wild character variety $\mathcal M_B(\Sigma,\mathcal C)$
is nonempty, then the dimension formula of
Boalch--Yamakawa~\cite[Proposition~3]{BY20} gives
\[
\dim_\C\mathcal M_B(\Sigma,\mathcal C)=2-(d,d),
\]
for the corresponding wild character variety,
where $d$ is the dimension vector of the diagram.

In our case each core node has multiplicity $1$, hence $d=(1,1)$ and
\[
(d,d)=d^{T}Cd=(1,1)\begin{bmatrix}2&-3\\-3&2\end{bmatrix}\binom{1}{1}=-2,
\qquad\text{so}\qquad
\dim_\C \cM_B(\Sigma,\cC)=4.
\]

On the other hand, the common level sets $S=s$ and $J=j$
of the CRDW system~\eqref{eq:RDW} are generically one-dimensional.
A generic symplectic leaf $S=s$ of the original Poisson structure $H$ on $X=\Spec A_0$ has complex dimension two. The same CRDW flow is generated by $J/2$ with respect to $H$, and fixing $S=s$ fixes the irregular class.

This dimension comparison suggests that our leaf should arise from the wild moduli
(via its Dolbeault counterpart) as a Poisson slice and/or a coisotropic reduction.
It is natural to ask whether such a reduction can be realised by imposing DS-type
constraints at $\infty$ (e.g.\ fixing a nilpotent leading term), but we do not pursue
this here.
\end{rmk}
Remark~\ref{rmk:RDWQV} motivates the following conjecture.  Throughout this conjecture, the formal monodromy
conjugacy classes are fixed as in the preceding Remark~\ref{rmk:RDWQV}.
\begin{conj}[Wild nonabelian Hodge / Hitchin-type realisation]\label{COJ:CMH}
 Let $\Sigma=\mathbb{P}^{1}$ and consider the meromorphic connection
\[
\nabla = d - L(\lambda)\,d\lambda
\]
associated to system~\eqref{eq:RDW}. Let $\Theta$ denote the local irregular class at
$\infty$ determined by the exponential factors $\{\langle q_i\rangle\}$ of $\nabla$
(possibly after passing to a local ramified cover), in the sense of
Hukuhara--Turrittin--Levelt.

\smallskip
\noindent
\textbf{(Betti/Core diagram).}
Let $\mathcal M_B(\Sigma,\Theta)$ denote the corresponding
stable wild character variety, with the fixed formal
monodromy conjugacy classes omitted from the notation.
Assume that $\mathcal M_B(\Sigma,\Theta)$ is nonempty.
Then the associated core diagram is the two-vertex diagram with three edges
(equivalently, generalized Cartan matrix
\(
\begin{bmatrix}
2 & -3\\
-3 & 2
\end{bmatrix}
\)),
as predicted by the Boalch--Yamakawa rule applied to $\Theta$.

\smallskip
\noindent
\textbf{(Dolbeault/Hitchin-type realisation, including twisted types).}
The wild nonabelian Hodge space underlying $\mathcal M_B(\Sigma,\Theta)$ is expected
to admit an appropriate Dolbeault realisation
\[
\mathcal M_{\mathrm{Dol}}(\Sigma,\Theta)
\]
in the general wild framework---for example, in terms of meromorphic, filtered,
or parahoric Higgs-type objects carrying the irregular data $\Theta$ together
with a natural Hitchin-type integrable system.

\smallskip
\noindent
\textbf{(Centrally
reduced RDW as a reduction/slice).}
A generic symplectic leaf $S=s$ of the original Poisson structure $H$ on $X$, carrying the CRDW flow, is expected
to be birationally Poisson equivalent to a two-dimensional
symplectic slice in, or a two-dimensional coisotropic reduction
constructed within, a Zariski-open subset of a suitable
Dolbeault moduli space $\mathcal M_{\mathrm{Dol}}$ equipped
with a wild Hitchin-type integrable system.

Under this identification, a suitable locally defined function of the Hitchin Hamiltonians
restricts to the slice, or descends to the reduced space.
The resulting Hamiltonian is $J/2$ and generates the CRDW flow~\eqref{eq:RDW} with respect to the induced Poisson
structure. The associated spectral curve is
\[
y^2=g(\lambda).
\]
The choice of the appropriate irregular data and the
construction of the slice or reduction are part of
the conjecture.
\end{conj}

\begin{rmk}
In the special case when the polar part at $\infty$ is holomorphically gauge
equivalent to a diagonal one (e.g.\ regular semisimple leading term), the
Dolbeault/Hitchin realisation in Conjecture~\ref{COJ:CMH} should reduce to the classical
wild nonabelian Hodge picture of Biquard--Boalch. The point of Conjecture~\ref{COJ:CMH}
is to extend this picture to possibly twisted and/or ramified irregular types,
for which the appropriate Dolbeault objects are expected to be described in
Mochizuki's good filtered framework, possibly after passing to a local ramified
cover.
\end{rmk}

In work in preparation~\cite{Fang26RDW0}, we construct an auxiliary system associated with the RDW system, show that it is of even Mumford type in the sense of Vanhaecke, and prove that the RDW flow is obtained by restricting this auxiliary system to an invariant slice. Related questions on the induced Poisson structure, the associated $\tau$-structures, and higher-rank generalizations will be discussed elsewhere.

We also indicate some connections between our system~\eqref{eq:RDW} and Arakawa's 4d/2d program~\cite{A18,A18ICM,A19CA}.
Assuming Conjecture~\ref{COJ:CMH}, we may compare the spectral curve $y^2=g(\lambda)$
with the Seiberg--Witten curve in~\cite{X13GAD} (Xie's label). 

Comparing our curve \(y^2=g(\lambda)\) with Xie's normal form
\(x^2=z^N+u_2z^{N-2}+\cdots+u_N\) for \((A_1,A_{N-1})\), we obtain \(N=6\)
(up to rescaling and admissible reparametrizations), i.e.\ it is compatible
with Xie's \((A_1,A_5)\) normal form~\cite[(4.4)]{X13GAD}. In particular, our family corresponds to the
sub-locus \(u_2=u_4=u_5=0\) in the full deformation space.
\begin{conj}[Conditional 4d/2d prediction for the CRDW spectral curve]\label{COJ:4d2d}
Assume Conjecture~\ref{COJ:CMH}. Assume further that the above spectral-curve family,
equipped with the appropriate Seiberg--Witten differential,
is realized as a subfamily of the Seiberg--Witten geometry
of the rank-two four-dimensional $\mathcal N=2$
$(A_1,A_5)$ Argyres--Douglas theory $T$. Then the associated chiral
algebra \(\chi(T)\) is expected to be the simple Bershadsky--Polyakov algebra,
equivalently the simple subregular affine \(W\)-algebra attached to
\((\mathfrak{sl}_3,f_{\mathrm{sub}})\) at parent affine level \(-9/4\).
\end{conj}

\begin{rmk}
This is the standard $(A_1,A_5)$ candidate in the 4d/2d correspondence. At the universal level, the $H$-twisted Zhu algebra of
the affine $\mathcal W$-algebra attached to
$(\mathfrak{sl}_3,f_{\mathrm{sub}})$ is isomorphic to
the corresponding finite $\mathcal W$-algebra.
The associated graded algebra of the latter under the
Kazhdan filtration is the Poisson algebra
$ C[\Sf]$.
For the simple affine $\mathcal W$-algebra appearing above,
the $H$-twisted Zhu algebra is a quotient of this finite
$\mathcal W$-algebra.
Thus the full Slodowy slice enters this comparison
as the ambient Poisson variety at the universal level.
\end{rmk}

Using the ideas explored above, we obtain a natural chiralization of
classical $R$–matrices for (affine) classical $\W$–algebras.

\begin{dfn}\label{dfn:AWARM}
Let $X=\Sf$ and $Y:=T^*[1]\Sf$.
We fix an identification of Poisson vertex algebras
$\W^{\cl}(\fg,f)\;\cong\;\OO(X_\infty)$
so that
\[
\Spec\W^{\cl}(\fg,f)\;\cong\;X_\infty.
\]
By Theorem~\ref{Thm:CS}, the corresponding PVA structure on $X_\infty$
is encoded by a degree-$2$ element
\[
P\in\Gamma(Y,\pi_*\OO_{Y_\infty})_2
\]
satisfying $[\int P,\int P]=0$, where $[-,-]$ is the Schouten bracket from
Theorem~\ref{Thm:CS}.
We say that an element
\[
\Pi_R\in\Gamma(Y,\pi_*\OO_{Y_\infty})_2
\]
is a \emph{classical $R$–matrix} of $\W^{\cl}(\fg,f)$ if and only if
it satisfies the Maurer-Cartan equation
\[
d_{P}\Theta_R+\tfrac12[\Theta_R,\Theta_R]=0,
\qquad d_P:=[\int P,-],
\]
where $P$ is the Hamiltonian structure of the classical affine $\W$–algebra and $\Theta_R:=\int(\Pi_R-P)$.\\

Equivalently, $\Pi_R$ is a classical $R$-matrix if and only if
\[
[\int\Pi_R,\int\Pi_R]=0.
\]

\end{dfn}

In this formalism, a classical $R$–matrix for $\W^{\cl}(\fg,f)$ may be
viewed as a QP–chiralization of the classical $R$–matrix on the finite
classical $\W$–algebra $\W^{\fin}(\fg,f)$ defined in
Definition~\ref{dfn:CFWRM}.

We now construct the natural chiralization of the centrally reduced RDW system in Example~\ref{exa:CDSsl3}. Recall that the classical finite $\W$-bracket is obtained from the classical affine $\W$-bracket by setting $\lambda$ and all positive-order derivatives to zero. We take the canonical jet extension after imposing the central constraint $c=0$. Set $X:=\Spec A_0$, where $A_0=\W^{\cl}_{\fin}(\mathfrak{sl}_3,f_{\mathrm{sub}})/(c)$, and $Y:=T^*[1]X$, with arc spaces $X_\infty$ and $Y_\infty$. Since $c$ is Poisson-central, the jet PVA descends to
\[
\V:=\C[X_\infty]\cong\C[(\Sf){_\infty}]/(\partial^n c\mid n\geq0)=\C[u_i^{(n)}\mid i=1,2,3,\ n\geq0],
\qquad u_i^{(n)}:=\partial^n u_i,
\]
whose bracket is determined by
\[
\{u_1{}_\lambda u_2\}_0=-2u_1,\qquad
\{u_1{}_\lambda u_3\}_0=u_2^2,\qquad
\{u_2{}_\lambda u_3\}_0=-2u_3.
\]
Let $M:=\bigoplus_{i=1}^3\C[\partial]u_i$. For $R\in\operatorname{End}_{\C[\partial]}(M)$, the $R$-deformed bracket is determined by
\[
\{u_i{}_\lambda u_j\}_R=\{R(u_i){}_\lambda u_j\}_0+\{u_i{}_\lambda R(u_j)\}_0.
\]
For constant $R$, this is the canonical jet extension of the finite $R$-deformation.

\begin{prop}[Constant compatible deformations]
Let
\begin{equation}\label{eq:CPTRM}
R=\begin{bmatrix}
\alpha & 0 & p\\
0 & \beta & 0\\
r & 0 & \alpha
\end{bmatrix},
\qquad \alpha,\beta,p,r\in\C,
\end{equation}
with the convention $R(u_i)=\sum_jR_{ij}u_j$. The brackets
\[
\begin{aligned}
\{u_1{}_\lambda u_2\}_R&=-2(\alpha+\beta)u_1+2pu_3,\\
\{u_1{}_\lambda u_3\}_R&=2\alpha u_2^2,\\
\{u_2{}_\lambda u_3\}_R&=2ru_1-2(\alpha+\beta)u_3
\end{aligned}
\]
define a PVA structure compatible with $\{\cdot{}_\lambda\cdot\}_0$. We call these constant $R$-deformations.
\end{prop}

We next consider first-order differential $R$-deformations and their Hamiltonian flows. Following the generalized Adler--Kostant--Symes (GAKS) scheme~\cite{Fang25}, take
\[
S=u_1u_3+\frac16u_2^3,\qquad
\cH=\int h(S,S_x,S_{xx},\ldots)\,dx,\qquad
q:=\frac{\delta\cH}{\delta S}.
\]
The differential algebra generated by $S$ is central for the original bracket, and
\[
\frac{\delta\cH}{\delta u}=q\begin{pmatrix}u_3\\ u_2^2/2\\ u_1\end{pmatrix}.
\]
We use the convention $u_{i,t}=\left.\{h{}_\lambda u_i\}_R\right|_{\lambda=0}$.

\begin{prop}[First-order deformations and Hamiltonian flows]\label{prop:R4PVA}\label{prop:R7PVA}
The following statements hold for complex parameters.
\begin{enumerate}
\item Let
\[
R(u_1)=R(u_2)=0,\qquad R(u_3)=au_1+bu_1^{(1)}+\gamma u_2.
\]
The induced bracket is a PVA bracket if and only if $b\gamma=0$, and is compatible with the original bracket if and only if $b=\gamma=0$. For $\gamma=0$, the Hamiltonian flow of $\cH$ is
\[
u_{1,t}=0,\qquad u_{2,t}=-2u_1Q,\qquad u_{3,t}=u_2^2Q,\qquad Q:=au_1q-b\partial(u_1q).
\]

\item Let
\[
R(u_1)=\alpha u_2+\beta u_3+\gamma u_3^{(1)},\qquad R(u_2)=R(u_3)=0.
\]
The induced bracket is a PVA bracket if and only if $\alpha\gamma=0$, and is compatible with the original bracket if and only if $\alpha=\gamma=0$. For $\alpha=0$, the Hamiltonian flow of $\cH$ is
\[
u_{1,t}=-u_2^2Q,\qquad u_{2,t}=2u_3Q,\qquad u_{3,t}=0,\qquad Q:=\beta u_3q-\gamma\partial(u_3q).
\]
\end{enumerate}
Both displayed flows satisfy
\[
S_t=u_3u_{1,t}+\frac12u_2^2u_{2,t}+u_1u_{3,t}=0.
\]
Thus local functionals with densities in $\C[S,\partial S,\ldots]$ pairwise Poisson-commute for the corresponding deformed bracket, and their Hamiltonian flows commute.
\end{prop}

For example, $h=S,\frac12S^2,SS_{xx}$ give $q=1,S,2S_{xx}$, respectively. Finally, certain diagonal $R$-deformations exhibit a decoupling property; a related phenomenon will appear in~\cite{FI26COU}.

\begin{thm}\label{thm:CRMDIA}
Let $R_\pm=\rho\,\operatorname{diag}(1,\pm1,-1)$, where $\rho\in\C$. Then $R_\pm^2=\rho^2\operatorname{Id}$, and the corresponding deformed PVA brackets have, up to skewsymmetry, the respective nonzero generator brackets
\[
\{u_1{}_\lambda u_2\}_+=-4\rho u_1,\qquad
\{u_2{}_\lambda u_3\}_-=4\rho u_3.
\]
For $\rho\neq0$, both PVAs are isomorphic to $\operatorname{Sym}(L)$, where $L:=\operatorname{Cur}(\mathfrak{aff}_1\oplus\C z)$, $\mathfrak{aff}_1=\C h\oplus\C e$ with $[h,e]=e$, and $z$ is central. Their generator Lie conformal algebras decompose into ideals as
\[
\begin{aligned}
L_+&=\Span_{\C[\partial]}\{u_1,u_2\}\oplus\C[\partial]u_3,\\
L_-&=\C[\partial]u_1\oplus\Span{\C[\partial]}\{u_2,u_3\},
\end{aligned}
\]
with central rank-one summands. Explicitly, the $\C[\partial]$-linear isomorphisms $L\to L_\pm$ are
\[
(h,e,z)\longmapsto
\begin{cases}
((4\rho)^{-1}u_2,u_1,u_3),&\text{for }L_+,\\
((4\rho)^{-1}u_2,u_3,u_1),&\text{for }L_-,
\end{cases}
\]
and extend to PVA isomorphisms.
\end{thm}

\end{document}